\documentclass{amsart}
\usepackage{amssymb}
\usepackage{amsmath}
\usepackage{latexsym}
\usepackage{eepic}
\usepackage{epsfig}
\usepackage{graphicx}
\usepackage{amscd}
\usepackage{eufrak}
\usepackage{mathrsfs}
\usepackage[english]{babel}
\usepackage[dvips]{color}

\newtheorem{theorem}[equation]{Theorem}
\newtheorem{theorem-definition}[equation]{Theorem-Definition}
\newtheorem{lemma-definition}[equation]{Lemma-Definition}
\newtheorem{definition-prop}[equation]{Proposition-Definition}

\newtheorem{prop}[equation]{Proposition}
\newtheorem{lemma}[equation]{Lemma}
\newtheorem{cor}[equation]{Corollary}
\newtheorem{definition}[equation]{Definition}

\newcommand{\llbracket}{[\negthinspace[}
\newcommand{\rrbracket}{]\negthinspace]}

\newcommand{\llpar}{(\negthinspace(}
\newcommand{\rrpar}{)\negthinspace)}

\theoremstyle{definition}
\newtheorem{example}[equation]{Example}

\newtheorem{remark}[equation]{Remark}

\newcommand{\N}{\ensuremath{\mathbb{Z}_{\geq 0}}}
\newcommand{\Z}{\ensuremath{\mathbb{Z}}}
\newcommand{\Q}{\ensuremath{\mathbb{Q}}}

\newcommand{\R}{\ensuremath{\mathbb{R}}}
\newcommand{\C}{\ensuremath{\mathbb{C}}}

\newcommand{\A}{\ensuremath{\mathbb{A}}}

\newcommand{\cX}{\ensuremath{\mathscr{X}}}
\newcommand{\mX}{\ensuremath{\mathfrak{X}}}

\newcommand{\cU}{\ensuremath{\mathscr{U}}}

\newcommand{\cY}{\ensuremath{\mathscr{Y}}}
\newcommand{\cZ}{\ensuremath{\mathscr{Z}}}

\renewcommand{\R}{\ensuremath{\mathbb{R}}}
\renewcommand{\C}{\ensuremath{\mathbb{C}}}

\renewcommand{\A}{\ensuremath{\mathbb{A}}}

\renewcommand{\cU}{\ensuremath{\mathscr{U}}}

\renewcommand{\cZ}{\ensuremath{\mathscr{Z}}}
\renewcommand{\cY}{\ensuremath{\mathscr{Y}}}

\newcommand{\Spec}{\ensuremath{\mathrm{Spec}\,}}
\newcommand{\Spf}{\ensuremath{\mathrm{Spf}\,}}

\newcommand{\Proj}{\ensuremath{\mathbb{P}}}

\newcommand{\red}{\mathrm{red}}

\newcommand{\an}{\mathrm{an}}

\newcommand{\weight}{\mathrm{wt}}

\newcommand{\lct}{\mathrm{lct}}
\newcommand{\Sk}{\mathrm{Sk}}

\numberwithin{equation}{subsection}

\newcommand{\sss}{\vspace{5pt} \subsubsection*{ }\refstepcounter{equation}{{\bfseries(\theequation)}\ }}

\begin{document}
\title[Weight functions on non-archimedean spaces]{Weight functions on non-archimedean analytic spaces and the Kontsevich-Soibelman skeleton}


\thanks{2010\,\emph{Mathematics Subject Classification}.
 Primary 14G22; secondary 13A18, 14F17.
 \newline
 The first-named author was partially supported by
 NSF grant DMS-1068190 and
  a Packard Fellowship. The second-named author was partially supported by the Fund for Scientific Research - Flanders (G.0415.10).}

\author{Mircea Musta\c{t}\u{a}}
\address{University of Michigan\\ Department of Mathematics\\ Ann Arbor, MI48109\\USA}
\email{mmustata@umich.edu}

\author[Johannes Nicaise]{Johannes Nicaise}
\address{KU Leuven\\
Department of Mathematics\\
Celestijnenlaan 200B\\3001 Heverlee \\
Belgium}
\email{johannes.nicaise@wis.kuleuven.be}

\begin{abstract}
We associate a weight function to pairs $(X,\omega)$ consisting of
a smooth and proper variety $X$ over a complete discretely valued
field and a pluricanonical form $\omega$ on $X$. This weight
function is a real-valued function on the non-archimedean
analytification of $X$. It is piecewise affine on the skeleton of
any regular model with strict normal crossings of $X$, and
strictly ascending as one moves away from the skeleton. We apply
these properties to the study of the Kontsevich-Soibelman skeleton
of $(X,\omega)$, and we prove that this skeleton is connected when
$X$ has geometric genus one and $\omega$ is a canonical form on
$X$. This result can be viewed as an analog of the
Shokurov-Koll\'ar connectedness theorem in birational geometry.
\end{abstract}

\maketitle

\section{Introduction}
\subsection{Skeleta of non-archimedean spaces}
 An important property of Berkovich spaces over non-archimedean
 fields is that, in many geometric situations, they can be
 contracted onto some subspace with piecewise affine structure, a
 so-called skeleton. These skeleta are
 usually constructed by choosing appropriate formal models for
 the space, associating a simplicial complex to the reduction of
 the formal model and embedding its geometric realization into the
 Berkovich space.
  For instance, if $C$ is a smooth projective
 curve of genus $\geq 2$ over an algebraically closed non-archimedean field, then it
 has a canonical skeleton, which is homeomorphic to the dual graph
 of the special fiber of its stable model \cite[\S4]{berk-book}. Likewise, if $(X,x)$ is
 a normal surface singularity over a perfect field $k$,
 then one can endow $k$ with its trivial absolute value and construct a $k$-analytic punctured tubular neighbourhood of $x$ in $X$ by removing $x$ from
 the generic fiber of the formal $k$-scheme $\Spf
 \widehat{\mathcal{O}}_{X,x}$. This $k$-analytic space contains a canonical skeleton,
 homeomorphic to the product of $\R_{>0}$ with the dual graph of the exceptional divisor of the
 minimal log-resolution of $(X,x)$ \cite{thuillier}.

  In higher dimensions, one can
  still associate skeleta to models with normal crossings (or to so-called pluristable models as in \cite{berk-contract}).
  However, it is no longer clear how to construct a canonical
  skeleton, since we usually cannot find a canonical model
  with normal crossings. The aim of this paper is to show how one can identify certain
  essential pieces that must appear in every skeleton by means of
  weight functions associated to pluricanonical forms.

\subsection{The work of Kontsevich and Soibelman}
  One of the starting points of this work was the following construction by
  Kontsevich and Soibelman \cite{KS}. Let $X$ be a smooth
  projective family of varieties over a punctured disc around the origin of the complex
  plane, and let $\omega$ be a relative differential form of
  maximal degree on $X$. Let $t$ be a local coordinate on $\C$ at $0$. Kontsevich and Soibelman defined a
  skeleton in the analytification of $X$ over $\C\llpar t\rrpar$ by taking
  the closure of the set of divisorial valuations that satisfy a
  certain minimality property with respect to the differential
  form $\omega$.
  Their goal was to find a non-archimedean
  interpretation of Mirror Symmetry. They
  studied in detail the maximally unipotent semi-stable
  degenerations of $K3$-surfaces, in which case the skeleton is
  homeomorphic to a two-dimensional sphere, and described how a
  degeneration can be reconstructed from the skeleton, equipped
  with a certain affine structure. This idea was further
  developed by Gross and Siebert in their theory of toric
  degenerations, using tropical and logarithmic geometry instead
  of non-archimedean geometry; see for instance the survey paper
  \cite{gross}.

\subsection{The weight function}
 We generalize the definition of the Kontsevich-Soibelman skeleton to
 smooth varieties $X$ over a complete discretely valued
 field $K$, endowed with a non-zero pluricanonical form
 $\omega$. We define the weight of $\omega$ at a divisorial point
 of $X^{\an}$ (that is, a point corresponding to a divisorial
 valuation on the function field of $X$) and define the skeleton
 $\Sk(X,\omega)$ of the pair $(X,\omega)$ as the closure of the set of divisorial
 points with minimal weight. A precise definition is given in
 \eqref{sss-KSsk} and it is compared to the one of
 Kontsevich and Soibelman in \eqref{sss-compar}.
  Next, we show how the weight function
 can be extended to the Berkovich skeleton associated to any
 regular model of $X$ whose special fiber has strict normal
 crossings, and even to the entire space $X^{\an}$ if $K$ has
 residue characteristic zero or $X$ is a curve. A remarkable
 property of this weight function is that it is piecewise affine on the Berkovich
 skeleton $\Sk(\cX)$ associated to any proper regular model $\cX$ with strict normal
 crossings, and
 strictly
 descending under the retraction from $X^{\an}$ to $\Sk(\cX)$ (Proposition \ref{prop-weightext}).
 We use this property to show that $\Sk(X,\omega)$ is the union of the faces of the Berkovich skeleton of $\cX$
 on which the weight function is constant and of minimal value (Theorem \ref{thm-KS}).
This generalizes Theorem 3 in \cite[\S6.6]{KS}; whereas the proof
in that paper relied on the Weak
 Factorization theorem, we only use elementary computations on
 divisorial valuations and approximation of arbitrary points on
 $X^{\an}$ by divisorial points.

\subsection{The connectedness theorem}
  Besides the construction of the weight function, the other main
  result of this paper is a connectedness theorem for the skeleton $\Sk(X,\omega)$ of a smooth and
  proper $K$-variety $X$ of geometric genus one (for instance, a Calabi-Yau variety)
  endowed with a
  non-zero differential form of maximal degree $\omega$. This skeleton does not depend on $\omega$, since
  $\Sk(X,\omega)$ is invariant under
  multiplication of $\omega$ by a non-zero scalar. We show that
  $\Sk(X,\omega)$ is always connected if $K$ has residue characteristic
  zero (Theorem \ref{thm-connected}). Our proof is based on a variant of Koll\'ar's
  torsion-free
  theorem for schemes over power series rings (Theorem \ref{thm-torfree}), which we deduce
  from the torsion-free theorem for complex varieties by means of Greenberg
  approximation.

\subsection{The relation with birational geometry}
   All these constructions and results have natural analogs in
   the birational geometry of complex varieties, replacing the pair
   $(X,\omega)$ by a smooth complex variety $Y$ equipped with a
   coherent ideal sheaf $\mathcal{I}$ and regular models with normal crossings
   by log resolutions. In particular, one can define in a very
   similar way a weight function on the non-archimedean punctured
   tubular neighbourhood of the zero locus of $\mathcal{I}$ in
   $X$; this is explained in Section \ref{sec-birat}, together with the close
    relation with the constructions in \cite{BFJ0} and \cite{jonsson-mustata}.
    If $h\colon Y'\to Y$ is a log resolution of $\mathcal{I}$ and $E$ is an irreducible component of
    the zero locus $Z(\mathcal{I}\mathcal{O}_{Y'})$ of $\mathcal{I}\mathcal{O}_{Y'}$, then the
   value of the weight function at a divisorial point associated to $E$ is equal
   to $\mu/N$, where $\mu-1$ is the multiplicity of $E$ in
   the relative canonical divisor $K_{Y'/Y}$, and $N$ is the
   multiplicity of $E$ in $Z(\mathcal{I}\mathcal{O}_{Y'})$. This is a
   classical invariant in birational geometry, closely related to the
   log discrepancy of $(X,\mathcal{I})$ at $E$, and its infimum
   over all possible log resolutions $h$ and divisors $E$ is the
   log canonical threshold of $(X,\mathcal{I})$.
   The
   counterpart of the Kontsevich-Soibelman skeleton coincides with the dual
   complex of the union of irreducible components of  $Z(\mathcal{I}\mathcal{O}_{Y'})$
   that compute the log canonical threshold, and our connectedness
   theorem translates into the connectedness theorem of Shokurov and
   Koll\'ar; see Section \ref{subsec-KSsk-birat}. Of course, the latter result was our main source of inspiration for the
   proof of the connectedness theorem for smooth and proper
   $K$-varieties of geometric genus one.

\subsection{Further questions}
 As we have seen, the skeleton of a pair $(X,\omega)$ as above is
 contained in the Berkovich skeleton of every regular proper model with
 strict normal crossings. We define the essential skeleton of $X$
 as the union of the skeleta $\Sk(X,\omega)$ as $\omega$
 runs through the set of non-zero pluricanonical forms on $X$ (Section \ref{ss-essential}).
 It would be quite interesting to know if
 one can define a suitable class of models of $X$ whose skeleta coincide with the essential skeleton (assuming that the Kodaira dimension
 of $X$ is non-negative). We are investigating this question in an ongoing
 project.

\subsection{Structure of the paper}
 To conclude this introduction, we give a brief survey of the
 structure of the paper. In Section \ref{sec-prelim} we introduce
 divisorial and monomial points on analytic spaces and we prove
 some basic properties. In Section \ref{sec-Berk} we explain the
 construction of the Berkovich skeleton associated to a regular
 model with strict normal crossings and we define its piecewise affine structure.
 This is a fairly straightforward generalization of the
 construction by Berkovich for pluristable formal schemes, but
 since the non-semistable case is not covered by the existing
 literature, we include some details here. The main new result is
 Proposition \ref{prop-domin}, which says that every proper morphism of models gives
  rise to an inclusion of skeleta. The core of the paper is
  Section \ref{sec-weight}, where we construct the weight
  function and prove its fundamental properties; see in particular Proposition
  \ref{prop-weightext}. The weight function is constructed in several steps, first defining it on divisorial and monomial points and then extending it to the entire
  analytic space by approximation. The applications to the
  Kontsevich-Soibelman skeleton are discussed in Section
  \ref{subsec-KS}. In Section \ref{sec-connected} we deduce a variant of Koll\'ar's
  torsion-free theorem for schemes over formal power series and use it to
  prove our
  connectedness theorem for skeleta of varieties of geometric
  genus one. Finally, in Section \ref{sec-birat}, we sketch the
  analogs of these results in the setting of complex birational geometry.

\subsection{Acknowledgements} We are grateful
to Olivier Wittenberg for helpful discussions concerning Section
\ref{subsec-Greenberg}, to Mattias Jonsson for pointing out a
 mistake in an earlier version of the paper and to Jenia Tevelev for pointing out the importance of
 working with pluricanonical forms instead of only canonical forms. Part of this research
has been done during the first-named author's visit to Leuven. He
is grateful to KU Leuven for making this visit possible.

\subsection{Notation}
\sss  We denote by $R$ a complete discrete valuation ring
with residue field $k$ and quotient field $K$.
 We denote by  $\frak{m}$ the maximal ideal in $R$ and by $v_K$ the valuation
$$K^{\times}\to \Z,$$ normalized in such a way that $v_K$ is surjective.
We define an absolute value $|\cdot|_K$ on $K$ by setting
$$|x|_K=\exp(-v_K(x))$$ for all $x$ in $K^{\times}$. This absolute
value turns $K$ into a complete non-archimedean field.

\sss We will  consider the special fiber functor
$$(\cdot)_k\colon (R-\mathrm{Sch})\to (k-\mathrm{Sch}),\, \cX\mapsto
\cX_k=\cX\times_R k$$ 
from the category of $R$-schemes to the category of $k$-schemes,
as well as the generic fiber functor
$$(\cdot)_k\colon (R-\mathrm{Sch})\to (K-\mathrm{Sch}),\, \cX\mapsto
\cX_K=\cX\times_R K$$ from the category of $R$-schemes to the
category of $K$-schemes.

\section{Monomial points on $K$-analytic spaces}\label{sec-prelim}
\subsection{Birational points}
\sss Let $X$ be an integral separated $K$-scheme of finite type.
Its analytification $X^{\an}$ is endowed with a canonical morphism
of locally ringed spaces
$$i\colon X^{\an}\to X.$$ For every point $x$ of $X^{\an}$, we denote by
$K(x)$ the residue field of $X$ at $i(x)$. This field carries a
natural valuation, and the residue field $\mathscr{H}(x)$ of
$X^{\an}$ at $x$ is the completion of $K(x)$ with respect to this
valuation.
 The dimension of the $\Q$-vector space
$$|\mathscr{H}(x)^{\times}|\otimes_{\Z} \Q$$ will be called the rational rank of $X$ at
$x$.


\sss We call a point of $X^{\an}$ birational if its image in $X$
is the generic point of $X$. We will denote the set of birational
points of $X^{\an}$ by $\mathrm{Bir}(X)$. It is clear that every
birational morphism of integral $K$-varieties $Y\to X$ induces a
bijection $\mathrm{Bir}(Y)\to \mathrm{Bir}(X)$.  For every
birational point $x$ on $X^{\an}$ we can define a real valuation
$v_x$ on the function field $K(X)$ of $X$ by
$$v_x\colon K(X)^{\times}\to \R,\,f\mapsto -\ln |f(x)|.$$
The map $x\mapsto v_x$ is a bijection between $\mathrm{Bir}(X)$
and the set of real valuations $K(X)^{\times}\to \R$ that extend
the valuation $v_K$ on $K$.

\subsection{Models}
\sss \label{sss-model}
 Let $X$ be a normal integral separated $K$-scheme of finite type. An $R$-model for $X$
is a normal flat separated $R$-scheme of finite type $\cX$ endowed
with an isomorphism of $K$-schemes $\cX_K\to X$. Note that we do
not impose any properness conditions on $X$ or $\cX$. If $\cX$ and
$\cY$ are $R$-models of $X$, then a morphism of $R$-schemes
$h\colon \cX\to \cY$ is called a morphism of $R$-models of $X$ if
$h_K\colon \cY_K\to \cX_K$ is an isomorphism that commutes with the
isomorphisms $\cX_K\to X$ and $\cY_K\to X$. Thus there exists at
most one morphism of $R$-models $\cY\to \cX$. If it exists, then
we say that $\cY$ dominates $\cX$.

\sss For every $R$-scheme $\cX$, we denote by $\widehat{\cX}$ its
$\frak{m}$-adic formal completion. If $\cX$ is an $R$-model of
$X$, then $\widehat{\cX}$ is a  flat separated formal $R$-scheme
of finite type, and we can consider its generic fiber
$\widehat{\cX}_\eta$ in the category of $K$-analytic spaces. This
is a compact analytic domain in the analytification $X^{\an}$ of
$X$. A point $x$ of $X^{\an}$ belongs to $\widehat{\cX}_\eta$ if
and only if the morphism $$\Spec \mathscr{H}(x) \to X$$ extends to
a morphism
$$\Spec \mathscr{H}(x)^o \to \cX,$$ where $\mathscr{H}(x)^o$
denotes the valuation ring of the valued field $\mathscr{H}(x)$.
Thus $X^{\an}=\widehat{\cX}_\eta$ if $\cX$ is proper over $R$. The
generic fiber $\widehat{\cX}_\eta$ is endowed with an
anti-continuous reduction map
$$\red_{\cX}\colon \widehat{\cX}_\eta\to \cX_k$$ that sends a point $x$
to the image of the closed point of $\Spec \mathscr{H}(x)^o$ under
the morphism $$\Spec \mathscr{H}(x)^o \to \cX.$$ In particular, a
birational point $x$ of $X^{\an}$ belongs to
$\widehat{\mathscr{X}}_\eta$ if and only if the valuation $v_x$
has a center on $\cX$, and in that case, the reduction map
$\red_{\cX}$ sends $x$ to the center of $v_x$.

\sss Assume that $\cX$ is a regular integral separated $R$-scheme
of finite type. Let $x$ be a point in $\widehat{\mathscr{X}}_\eta$
and
 let $D$ be a divisor on $\cX$ whose support does not
 contain $i(x)$. Then we set
  $$v_x(D)=-\ln|f(x)|$$ where $f$ is any element of
  $K(X)^{\times}$ such that, locally at $\red_{\cX}(x)$, we have
  $D=\mathrm{div}(f)$.
 It is obvious that the function $v_x(\cdot)$ is $\Z$-linear in $D$ and that is behaves well with respect to
pull-backs: if $h\colon \cX'\to \cX$ is a proper morphism of regular
$R$-models of $X$, then $v_x(D)=v_x(h^*D)$ for every divisor $D$
on $\cX$.

\begin{prop}\label{prop-cont} Assume that $\cX$ is a regular integral separated $R$-scheme of finite
type. Then for every divisor $D$ on $\cX$, the function
$$\{x\in \widehat{\cX}_\eta\,|\,i(x)\notin \mathrm{Supp}(D)\} \to \R,\,x\mapsto v_x(D)$$ is
continuous.
\end{prop}
\begin{proof}
 For every open subscheme $\cY$ of $\cX$, the
 space $\widehat{\cY}_\eta$ is a closed analytic domain of
 $\widehat{\cX}_\eta$ by anti-continuity of the reduction map
 $$\red_{\cX}\colon \widehat{\cX}_\eta\to \cX_k$$
and the fact that $\widehat{\cY}_\eta=\red_{\cX}^{-1}(\cY_k)$
\cite[\S1]{berk-vanish}.
 Thus we may assume that $D=\mathrm{div}(f)$ for some rational
 function $f$ on $\cX$. Then
 $$v_x(D)=-\ln |f(x)|$$ is a continuous function
 in $x$.
\end{proof}

\sss \label{sss-sncd} Let $X$ be a connected regular separated
$K$-scheme of finite type. An $sncd$-model for $X$ is a regular
$R$-model $\cX$ such that $\cX_k$ is a divisor with strict normal
crossings. Again, we do not impose any properness conditions on
$X$ or $\cX$; for instance, according to our definition, $X$ is an
$sncd$-model of itself. If $X$ is proper and either $k$ has
characteristic zero or $X$ is a curve, then every proper $R$-model
of $X$ can be dominated by a proper $sncd$-model, by Hironaka's
resolution of singularities. If $k$ has positive characteristic
and $X$ has dimension at least $2$, the existence of a proper
$sncd$-model is not known.

\subsection{The Zariski Riemann space}\label{subsec-ZR}
\sss Let $X$ be a normal integral proper $K$-scheme. We denote by
$\mathcal{M}_X$ the category of proper $R$-models $\cX$ of $X$,
where the morphisms are morphisms of $R$-models. We can dominate
any pair of proper $R$-models $\cX$, $\cX'$ by a common proper
$R$-model of $X$, and we have already observed that there exists
at most one morphism of $R$-models $\cX'\to \cX$. These properties
imply that the category $\mathcal{M}_X$ is cofiltered.
 The Zariski Riemann space of $X$ is defined as the limit
 $$X^{\mathrm{ZR}}=\lim_{\stackrel{\longleftarrow}{\cX\in \mathcal{M}_X}}\cX_k$$ in the category of locally ringed spaces.

\begin{prop}\label{prop-ZR} The map $j\colon X^{\an}\to X^{\mathrm{ZR}}$ induced by the
reduction maps $\red_{\cX}\colon X^{\an}\to \cX_k$ is injective. It has
a continuous retraction $r\colon X^{\mathrm{ZR}}\to X^{\an}$ such that
the topology on $X^{\an}$ is the quotient topology with respect to
 $r$ and such that, for every point $x$ in $X^{\mathrm{ZR}}$ and every proper $R$-model $\cX$ of $X$, the image of $x$ under the natural
  projection $p_{\cX}\colon X^{\mathrm{ZR}}\to\cX_k$ lies in the closure of $\{\red_{\cX}\circ
  r(x)\}$.
\end{prop}
\begin{proof}
 Combining Theorems 3.4 and 3.5 in \cite{schneider-vdput},
  we obtain a natural surjection $r\colon X^{\mathrm{ZR}}\to X^{\an}$ with all the required properties. To be precise, the results in
 \cite{schneider-vdput} are formulated for affinoid rigid
 $K$-varieties, but as noted at the end of \cite{schneider-vdput},
 they extend immediately to quasi-separated and quasi-compact
 rigid $K$-varieties. To pass to our algebraic setting, it then suffices to observe that the
 algebraizable formal $R$-models of $X^{\an}$ are cofinal in the
 category of admissible formal $R$-models: if $\cX$ is a proper $R$-model of $X$ then by Raynaud's theorem \cite[4.1]{BL1} the admissible
 blow-ups of the $\frak{m}$-adic completion $\widehat{\cX}$ of $\cX$ are cofinal
 in the category of admissible formal $R$-models of $X$.
 The center of such a blow-up is a closed subscheme of
 $$\widehat{\cX}\times_{R}(R/\frak{m}^n)\cong
 \cX\times_{R}(R/\frak{m}^n)$$
 for some $n>0$.  In particular, admissible
 blow-ups are algebraizable (simply blow up the corresponding
 closed subscheme of $\cX$).
\end{proof}

\begin{cor}\label{cor-ZR}
Let $Z$ be a closed subscheme of $X$, and let $x$ be any point of
$(X\smallsetminus Z)^{\an}$. Then there exists a proper $R$-model $\cX$
of $X$ such that the closure of $\red_{\cX}(x)$ in $\cX_k$ is
disjoint from the closure of $Z$ in $\cX$.
\end{cor}
\begin{proof}
 The fiber $r^{-1}(x)$ is closed in $X^{\mathrm{ZR}}$ and disjoint from
$j(Z^{\an})$. The closure of $\{j(x)\}$ in $X^{\mathrm{ZR}}$ is
the intersection of the sets $p_{\cX}^{-1}(C_{\cX})$ where $\cX$
runs through the proper $R$-models of $X$, $C_{\cX}$ denotes the
closure of $\red_{\cX}(x)$ in $\cX_k$ and
$p_{\cX}\colon X^{\mathrm{ZR}}\to \cX_k$ is the natural projection.
 Thus we can find a proper $R$-model $\cX$
of $X$ such that $\red_{\cX}^{-1}(C_{\cX})$ is disjoint from
$j(Z^{\an})$. If we denote by $\mathfrak{Z}$ the $\frak{m}$-adic
completion of the schematic closure of $Z$ in $\cX$, then
$\mathfrak{Z}_\eta=Z^{\an}$ by properness of $Z$ and the reduction
map $Z^{\an}\to \mathfrak{Z}_k$ is surjective because
$\mathfrak{Z}$ is $R$-flat (its image is closed by
\cite[p.542]{berk-vanish} and it contains all the closed points of
$\mathfrak{Z}_k$ by \cite[10.1.38]{liu}). It follows that
$C_{\cX}$ must be disjoint from the
 closure of $Z$ in $\cX$.
\end{proof}

\subsection{Divisorial and monomial points}
\sss \label{sss-divpointexpl} Let $X$ be a normal integral
separated $K$-scheme of finite type. Let $\cX$ be an $R$-model for
$X$ and let $E$ be an irreducible component of $\cX_k$. Denote by
$\xi$ the generic point of $E$. The local ring
$\mathcal{O}_{\cX,\xi}$ is a discrete valuation ring with fraction
field $K(X)$, the field of rational functions on $X$. We denote by
$v_E$ the associated discrete valuation on $K(X)$, normalized in
such a way that its restriction to $K$ coincides with $v_K$. This
is the unique valuation on $\cX$ that extends $v_K$ and is
centered at $\xi$. The index of $\Z=|K^{\times}|$ in the value
group of $v_E$ is equal to the multiplicity of $E$ in the divisor
$\cX_k$.
 The
valuation $v_E$ defines a birational point $x$ of $X^{\an}$ such
that $v_x=v_E$. We call this point the divisorial point associated
to the pair $(\cX,E)$. It is the unique point of
$\red_{\cX}^{-1}(\xi)$. We call the pair $(\cX,E)$ a divisorial
presentation for $x$.
 Removing a closed subset of $\cX_k$ that does not contain $E$ has
no influence on the divisorial valuation $v_E$, so that we can
always assume that $\cX$ is regular and that $\cX_k$ is
irreducible.

\sss \label{sss-mondef} The construction of divisorial points can
be generalized in the following way. Let $\cX$ be an $R$-model of
$X$ and let $(E_1,\ldots,E_r)$ be a tuple of distinct irreducible
components of $\cX_k$ such that the intersection
$$E=\cap_{i=1}^r E_i$$ is non-empty. Let $\xi$ be a generic point
of $E$, and assume that, locally at $\xi$, the $R$-scheme $\cX$ is
regular and $\cX_k$ is a divisor with strict normal crossings. In
this case, we call $(\cX,(E_1,\ldots,E_r),\xi)$ an $sncd$-triple
for $X$.
 Then there exist a regular system of local parameters
$z_1,\ldots,z_r$ in $\mathcal{O}_{\cX,\xi}$, positive integers
$N_1,\ldots,N_r$ and a unit $u$ in $\mathcal{O}_{\cX,\xi}$ such
that
$$\pi:=uz_1^{N_1}\cdot \ldots \cdot z_r^{N_r}$$ is a uniformizer in $R$
and such that, locally at $\xi$, the prime divisor $E_i$ is
defined by the equation $z_i=0$.

We say that an element $f$ of $\widehat{\mathcal{O}}_{\cX,\xi}$ is
 admissible if it has an expansion of the form
 \begin{equation}\label{eq-adm}
 f= \sum_{\beta\in
 \N^r}c_{\beta} z^{\beta}\end{equation}
 where each coefficient $c_{\beta}$ is either zero or a unit in
 $\widehat{\mathcal{O}}_{\cX,\xi}$. Such an expansion is called an
 admissible expansion of $f$.
Let $\alpha$ be an element of $\R_{\geq 0}^r$ such that $$\alpha_1
N_1+ \ldots + \alpha_r{N_r}=1.$$ For every $\beta$ in $\R^r$, we
set $$\alpha\cdot \beta=\sum_{i=1}^r \alpha_i\beta_i.$$

\begin{lemma}\label{lemm-admiss}
Let $A$ be a Noetherian local ring with maximal ideal $\frak{m}_A$
and residue field $\kappa_A$, and let $(y_1,\ldots,y_m)$ be a
system of generators for $\frak{m}_A$. We denote by $\widehat{A}$
the $\frak{m}_A$-adic completion of $A$. Let
 $B$ be a subring of $A$ such that the elements $y_1,\ldots,y_m$
 belong to $B$ and generate the ideal $B\cap \frak{m}_A$ in $B$.
 Then, in the ring $\widehat{A}$, every element $f$ of $B$ can be
written as
\begin{equation}\label{eq-expans} f=\sum_{\beta\in
\N^m}c_{\beta}y^{\beta}\end{equation} where the coefficients
 $c_{\beta}$ belong to $(A^{\times}\cap B)\cup\{0\}$.
\end{lemma}
\begin{proof}
Such an expansion for $f$ is constructed inductively in the
following way. Since $A$ is local, $f$ belongs either to
$A^{\times}$ or to the maximal ideal $\frak{m}_A$. In the latter
case, we can write $f$ as a $B$-linear combination of the elements
$y_1,\ldots,y_m$.
 Now suppose that $i$ is a positive
integer and that we can write every $f$ in $B$ as a sum of an
element $f_i$ of the form \eqref{eq-expans} and a $B$-linear
combination of degree $i$ monomials in the elements
$y_1,\ldots,y_m$. Applying this assumption to the coefficients of
the $B$-linear combination, we see that we can write $f$ as the
sum of an element $f_{i+1}$ of the form \eqref{eq-expans} and a
$B$-linear combination of degree $i+1$ monomials in the elements
$y_1,\ldots,y_m$ such that $f_i$ and $f_{i+1}$ have the same
coefficients in degree $<i$. Repeating this construction, we find
an expansion for $f$ of the form \eqref{eq-expans}.
\end{proof}

\begin{prop}\label{prop-qmon}
We keep the notations of \eqref{sss-mondef}.
\begin{enumerate}
\item \label{item:admiss} Every element of
$\widehat{\mathcal{O}}_{\cX,\xi}$ is admissible. \item
\label{item:monval} There exists a unique real valuation
$$v_{\alpha}\colon K(X)^{\times}\to \R$$ such that
$$v_{\alpha}(f)= \min\{\alpha\cdot \beta\,|\,\beta\in \N^r,\,c_{\beta}\neq 0 \} $$
for every  non-zero element $f$ in $\mathcal{O}_{\cX,\xi}$ and
every admissible expansion of $f$ of the form
 \eqref{eq-adm}. This valuation does not depend on the choice of
 $z_1,\ldots,z_r$, and its restriction to $K$ coincides with
 $v_K$.
 \end{enumerate}
\end{prop}
\begin{proof}
\eqref{item:admiss} This follows at once from Lemma
\ref{lemm-admiss}, taking $A=B=\widehat{\mathcal{O}}_{\cX,\xi}$.

\eqref{item:monval} Uniqueness is clear from \eqref{item:admiss}.
Let us show that our definition of the value $v_{\alpha}(f)$ does
not depend on the chosen admissible expansion of $f$. For
notational convenience, we set $A=\widehat{\mathcal{O}}_{\cX,\xi}$
and we denote by $\frak{m}_A$ and $\kappa_A$ its maximal ideal and
residue field, respectively.

  To every
admissible expansion \eqref{eq-adm} we can associate a Newton
polyhedron $\Gamma$. This is the convex hull of the set
$$\{\beta\in \N^r\,|\,c_\beta\neq 0\}+\R_{\geq 0}^r$$
in $\R^r$. We denote by $\Gamma_c$ the set of points in $\N^r$
that lie on a compact face of $\Gamma$ and we set
$$f_c=\sum_{\beta\in
\Gamma^c}\overline{c}_{\beta}z^{\beta}$$
 where $\overline{c}_\beta$ denotes the residue
class of $c_\beta$ in $\kappa$. Thus $f_c$ is an element of
$\kappa[z_1,\ldots,z_r]$. We claim that it only depends on $f$,
and not on the chosen admissible expansion \eqref{eq-adm}. To see
this, let $$f=\sum_{\beta\in \N^r}c'_{\beta}z^\beta$$ be another
admissible expansion of $f$, with associated set $\Gamma_c'$ and
polynomial $f'_c$.
 Then
$$\sum_{\beta\in \N^r}(c_\beta-c'_{\beta})z^\beta=0.$$ Choosing
admissible expansions for the elements $c_\beta-c'_{\beta}$ that
do not lie in $A^{\times}\cup \{0\}$, we
 can rewrite this expression into an admissible expansion
 $$0=\sum_{\beta\in \N^r}d_\beta z^\beta$$ such that
 $\overline{d}_\beta=\overline{c}_\beta-\overline{c}'_\beta$ for
 all $\beta$ in $\Gamma_c\cup \Gamma'_c$. Since the graded
 $\kappa$-algebra $\oplus_{i\geq 0} \frak{m}_A^i/\frak{m}_A^{i+1}$ of the local ring $A$ is isomorphic to the
 polynomial ring over $\kappa_A$ in the residue classes of
 $z_1,\ldots,z_d$ modulo $\frak{m}^2_A$, the elements
 $d_\beta$ must all be equal to zero. It follows that
 $\Gamma_c=\Gamma_c'$ and $f_c=f'_c$.

 Now we set
$$
m=\min\{\alpha\cdot \beta\,|\,\beta\in
\Gamma_c,\,\overline{c}_\beta\neq 0\}
$$
 and we
 denote by
$\Gamma^{\alpha}_c$ the set of points $\beta$ in $\Gamma_c$ such
that $\alpha\cdot \beta=m$. We set
$$f_{\alpha}=\sum_{\beta\in
\Gamma_c^{\alpha}}\overline{c}_{\beta}z^{\beta}\in
\kappa_A[z_1,\ldots,z_r].$$ Then $m$ and $f_{\alpha}$ are
completely determined by $f_c$, so that they do not depend on the
chosen admissible expansion for $f$. Moreover,
\begin{equation}\label{eq-minimum}
v_{\alpha}(f)=\min\{\alpha\cdot \beta\,|\,\beta\in
\Gamma_c,\,c_\beta\neq 0\}=m
\end{equation}
 because for every linear form $L$ on $\R^r$ with
non-negative coefficients, the minimal value of $L$ on $\Gamma$ is
reached on $\Gamma^c$.
 It follows that $v_{\alpha}(f)$ only depends on
$f$, and not on the chosen admissible expansion.

 Now it is easy to check that $v_{\alpha}$ is a valuation: if $f$
 and $g$ are non-zero elements in $\mathcal{O}_{\cX,\xi}$, then clearly
 $$v_{\alpha}(f+g)\geq \min\{v_{\alpha}(f),v_{\alpha}(g)\}$$
and we also have $v_{\alpha}(f\cdot
g)=v_{\alpha}(f)+v_{\alpha}(g)$ because $(f\cdot
g)_{\alpha}=f_{\alpha}\cdot g_{\alpha}$.
 It is
 obvious that $v_{\alpha}(f)$ does not depend on the choice of $z_1,\ldots,z_r$
 since these elements are determined up to a unit in
 $\mathcal{O}_{\cX,\xi}$. To see that $v_{\alpha}$ extends $v_K$,
 note that
 $$\pi=uz_1^{N_1}\ldots z_r^{N_r}$$ is an admissible expansion for
 the uniformizer $\pi$ in $R$ so that
 $$v_{\alpha}(\pi)=\sum_{i=1}^r \alpha_iN_i=1.$$
 \end{proof}

\sss The valuation $v_{\alpha}$ from Proposition \ref{prop-qmon}
extends the discrete valuation $v_K$ on $K$, and thus defines a
birational point $x$ of $\widehat{\cX}_\eta \subset X^{\an}$ such
that $v_x=v_{\alpha}$. We call this birational point the monomial
point associated to the $sncd$-triple $(\cX,(E_1,\ldots,E_r),\xi)$
and the tuple $\alpha$, and we will say that these data represent
the point $x$. Replacing $\cX$ by an open neighbourhood of $\xi$
has no influence on the associated monomial point $x$, so that we
can always assume that $\cX$ is regular and that $\cX_k$ is a
divisor with strict normal crossings.

\sss \label{sss-reducemon} Permuting the indices, we may assume
that there exists an element $r'$ in $\{1,\ldots,r\}$ such that
$\alpha_i\neq 0$ for all $i\in \{1,\ldots, r'\}$ and $\alpha_i=0$
for all $i\in \{r'+1,\ldots,r\}$. If we denote by $\xi'$ the
unique generic point of $\cap_{i=1}^{r'}E_i$ whose closure
contains $\xi$, then it is clear from the constructions that the
data
$$((E_1,\ldots,E_r),\xi,(\alpha_1,\ldots,\alpha_r))\mbox{ and
}((E_1,\ldots,E_{r'}),\xi',(\alpha_1,\ldots,\alpha_{r'}))$$ define
the same monomial point $x$ of $X^{\an}$.   Moreover, $\xi'$ is
the center of $v_{x}$, and
 $\alpha_i=v_x(E_i)$ for every
$i\in \{1,\ldots,r\}$. Thus $\xi'=\red_{\cX}(x)$ and $\xi'$,
$(E_1,\ldots,E_{r'})$ and $(\alpha_1,\ldots,\alpha_{r'})$ are
completely determined by the model $\cX$ and the monomial point
$x$, up to permutation of the indices $\{1,\ldots,r'\}$. It is
clear that the rational rank of $X$ at $x$ is the dimension of the
$\Q$-vector subspace of $\R$ generated by the coordinates of
$\alpha$.

\sss \label{sss-mondef2} We say that a point of $X^{\an}$ is
divisorial if it is the divisorial point associated to some
$R$-model $\cX$ of $X$ and some irreducible component of $\cX_k$.
We denote the set of divisorial points on $X^{\an}$ by
$\mathrm{Div}(X)$. The notion of a monomial point on $X^{\an}$ is
defined analogously, and the set of monomial points on $X^{\an}$
is denoted by $\mathrm{Mon}(X)$. Note that every divisorial point
is monomial. If $x$ is a monomial
 point on $X^{\an}$ and $\cX$ is an $R$-model of $X$, then we say that
 $\cX$ is adapted to $x$ if there exist irreducible components
 $(E_1,\ldots,E_r)$ of $\cX_k$, a generic point $\xi$ of
 $\cap_{i=1}^rE_i$ and a tuple $\alpha$ in $\R^r_{>0}$ such that
 $(\cX,(E_1,\ldots,E_r),\xi)$ is an $sncd$-triple and
 $x$ is the monomial point associated to this $sncd$-triple and
 the tuple $\alpha$.

\begin{prop}\label{prop-rank1}
Let $x$ be a monomial point of $X^{\an}$. Then the following are
equivalent:
\begin{enumerate}
\item The point $x$ is divisorial. \item The valuation $v_x$ is
discrete. \item The analytic space $X^{\an}$ has rational rank one
at $x$.
 \end{enumerate}
\end{prop}
\begin{proof}
It suffices to prove that $x$ is divisorial if $X^{\an}$ has
rational rank one at $x$.
 Let $(\cX,(E_1,\ldots,E_r),\xi)$ be an $sncd$-triple for $X$ and
 let $\alpha$ be an element of $\R^r_{>0}$ such that these data
 represent the point $x$. Since $X^{\an}$ has rational rank one at
 $x$, the tuple $\alpha$ must belong to $\Q^r_{>0}$.
 Permuting the indices, we may assume that $\alpha_1$ is minimal
 among the coordinates of $\alpha$. We consider the blow-up $h\colon \cX'\to \cX$ at the closure
 of $\xi$.
  We denote by $E'_i$ the strict transform of $E_i$, for
all $i$ in $\{2,\ldots,r\}$, and we denote by $E'_1$ the
exceptional divisor of the blow-up. Let $\xi'$ be the generic
point of $E'_1\cap\ldots \cap E'_r$.  Then a straightforward
computation shows that $$(\cX',(E'_1,\ldots,E'_r),\xi')$$ is still
an $sncd$-triple, and together with the tuple
$$\alpha'=(\alpha_1,\alpha_2-\alpha_1,\ldots,\alpha_r-\alpha_1)$$ it represents the point
$x$. By the construction in
 \eqref{sss-reducemon}, we can eliminate the zero entries in
 $\alpha'$. The rational
 rank of $X^{\an}$ at $x$ is equal to one, so that all the
 $\alpha_i$ are integer multiples of a common rational number
 $q>0$. An elementary induction argument shows that, starting from
 a finite tuple of positive integers and repeatedly choosing
 one of the minimal non-zero coordinates and subtracting it from
 the other coordinates, we arrive after a finite number of steps
 at a tuple with only one non-zero coordinate.
 Thus, repeating our blow-up procedure
 finitely many times, we arrive at a monomial presentation with $r=1$, that
 is, a divisorial presentation for $x$.
\end{proof}

\begin{prop}\label{prop-dense}
The set $\mathrm{Div}(X)$ of divisorial points on $X^{\an}$ is
dense in $X^{\an}$.
\end{prop}
\begin{proof}
 If $j\colon X\to \overline{X}$ is a normal compactification of $X$, then
 $j^{\an}\colon X^{\an}\to \overline{X}^{\an}$ is an open immersion of
 $K$-analytic spaces. It identifies the divisorial points on
 $X^{\an}$ and $\overline{X}^{\an}$: every $R$-model for $X$ can
 be extended to an $R$-model of $\overline{X}$ by gluing $\overline{X}$ to its generic
 fiber, which shows that the image of a divisorial point in
 $X^{\an}$ is divisorial in $\overline{X}^{\an}$.
 Conversely, if $x$ is a divisorial point on $\overline{X}^{\an}$
 associated to an $R$-model $\overline{\cX}$ of $\overline{X}$ and an irreducible
 component $E$ of $\overline{\cX}_k$, then by removing the
 closure of $\overline{X}\smallsetminus X$ we get a
 divisorial presentation of $x$ as a point of $X^{\an}$.
 Thus we may assume that $X$
 is proper over $K$.

 We denote by $X^{\mathrm{ZR}}$ the Zariski Riemann space of $X$ as defined in Section \ref{subsec-ZR}. Let $\cX$ be a proper $R$-model of $X$
 and denote by $p\colon X^{\mathrm{ZR}}\to \cX_k$ the projection morphism. Then
 for every generic point $\xi$ of $\cX_k$, the fiber $p^{-1}(\xi)$
 consists of a unique point, that we will denote by $\xi^{\mathrm{ZR}}$. To
 see this, note that every morphism of proper $R$-models $\cX'\to
 \cX$ is an isomorphism over an open neighbourhood of $\xi$ by
 Zariski's Main Theorem, because $\cX$ is normal. Since $X^{\mathrm{ZR}}$
 carries the limit topology, the points of the form $\xi^{\mathrm{ZR}}$
 form a dense subset of $X^{\mathrm{ZR}}$ as $\cX$ varies over the class of proper $R$-models of $X$. To conclude the proof, it
 suffices to observe that the image of $\xi^{\mathrm{ZR}}$ under the retraction $r\colon X^{\mathrm{ZR}}\to  X^{\an}$ is the divisorial point $x$ on $X^{\an}$
 associated to $\cX$ and the closure of $\xi$. Indeed,
 $p(\xi^{\mathrm{ZR}})=\xi$ must lie in the closure of $(\red_{\cX}\circ r)(\xi^{\mathrm{ZR}})$ by Proposition \ref{prop-ZR}. Thus
  $\red_{\cX}(r(\xi^{\mathrm{ZR}}))=\xi$, but $x$ is the only element in $\red_{\cX}^{-1}(\xi)$.
\end{proof}

\section{The Berkovich skeleton of an
$sncd$-model}\label{sec-Berk}
\subsection{Definition of the Berkovich skeleton}
\sss In \cite{berk-contract}, Berkovich associates a skeleton
 to a certain class of formal schemes $\mX$ over $R$,
the so-called pluristable formal $R$-schemes. This skeleton is a
closed subspace of the generic fiber $\mX_\eta$ of $\mX$, and
 Berkovich has shown
 that it is a strong deformation
retract of $\mX_\eta$. This construction was a crucial ingredient
of his proof of the local contractibility of smooth analytic
spaces over a non-archimedean field. The special case of strictly
semi-stable formal $R$-schemes is described in detail in
\cite[\S3]{ni-sing}. If $\cX$ is a regular separated $R$-scheme of
finite type such that $\cX_k$ is a divisor with strict normal
crossings, then its $\frak{m}$-adic completion $\widehat{\cX}$ is
not pluristable if $\cX_k$ is not reduced. Nevertheless, the
definition of the skeleton can be generalized to this setting in a
fairly straightforward way. In this section, we explain the
construction of the skeleton and we prove some basic properties
that we will need further on.

\sss Let $X$ be a connected regular separated $K$-scheme of finite
type and let $\cX$ be an $sncd$-model for $X$. Then the reduced
special fiber $(\cX_k)_{\red}$ is a strictly semi-stable
$k$-variety, to which one can associate an unoriented simplicial
set $\Delta(\cX_k)$ in a standard way: see for instance
\cite[\S3.1]{ni-sing}. If $\cX_k$ is a curve, then $\Delta(\cX_k)$
is simply the unweighted dual graph of $\cX_k$. The geometric
realization $|\Delta(\cX_k)|$ is a compact simplicial space whose
points can be uniquely represented by couples $(\xi,w)$ where
$\xi$ is a generic point of an intersection
 of $r$ distinct irreducible components of
$\cX_k$, for some $r>0$, and $w$ is an element of the open simplex
 $$\Delta^o_{\xi}=\{x\in \R^{\Psi(\xi)}_{>0}\,|\,\sum_{i\in \Psi(\xi)}x_i=1\}.$$
Here $\Psi(\xi)$ denotes the set of irreducible components of
$\cX_k$ that pass through $\xi$. In \cite[\S3.2]{ni-sing}, we
called $(\xi,w)$ the
 barycentric representation of the point, and $w$ the tuple of barycentric coordinates.

\sss \label{sss-Bsk} We define a map
$$\Sk_{\cX}\colon |\Delta(\cX_k)|\to X^{\an}$$
by sending a point $P$ with barycentric representation $(\xi,w)$
as above to the monomial point on $X^{\an}$ associated to the
$sncd$-triple $(\cX,(E_1,\ldots,E_r),\xi)$ and the tuple
$$\alpha=(\frac{w_{E_1}}{N_1},\ldots,\frac{w_{E_r}}{N_r})  \in \R^r_{>0}$$
where $E_1,\ldots,E_r$ are the irreducible components of $\cX_k$
 passing through $\xi$ and $N_1,\ldots,N_r$ are their multiplicities in $\cX_k$.
 We call the image of $\Sk_{\cX}$ the Berkovich skeleton of
$\cX$, and denote it by $\Sk(\cX)$. We endow  it with the induced
topology from $X^{\an}$. Beware that it not only depends on $X$,
but also on the chosen $sncd$-model $\cX$. Note that the Berkovich
skeleton of $\cX$ is contained in $\widehat{\cX}_\eta$, and that a
monomial point on $X^{\an}$ lies on $\Sk(\cX)$ if and only if the
$sncd$-model $\cX$ is adapted to $x$,  in the sense of
\eqref{sss-mondef2}. It is clear from the definition that, for
every open subscheme $\cY$ of $\cX$, the skeleton $\Sk(\cY)\subset
\cY_K^{\an}$ is equal to the closed subset
$\red_{\cX}^{-1}(\cY_k)$ of $\Sk(\cX)$.

\begin{prop}\label{prop-Bskhomeo}
The map $\Sk_{\cX}\colon |\Delta(\cX_k)|\to \Sk(\cX)$ is a
homeomorphism.
\end{prop}
\begin{proof}
We have seen in \eqref{sss-reducemon} that $\xi$ and $w$ are
completely determined by $\cX$ and $\Sk_{\cX}(P)$. Thus
$\Sk_{\cX}$ is injective. Since $|\Delta(\cX_k)|$ is compact and
$X^{\an}$ is Hausdorff, it suffices to show that $\Sk_{\cX}$ is
continuous.

 Let $\cU$ be an open subscheme of $\cX$, and denote by
 $\Delta(\cU_k)$ the simplicial set associated to the
 semi-stable $k$-variety $(\cU_k)_{\red}$. Then $\widehat{\cU}_\eta$ is a closed analytic domain in $\widehat{\cX}_\eta$. The open immersion
 $\cU_k\to \cX_k$ induces a closed embedding $|\Delta(\cU_k)|\to
 |\Delta(\cX_k)|$, and
 $\Sk_{\cX}$ coincides with $\Sk_{\cU}$ on $|\Delta(\cU_k)|$. If we cover $\cX$ by finitely many affine open subschemes $\cU$, then
 the spaces $|\Delta(\cU_k)|$ will form a closed cover of $|\Delta(\cX_k)|$.
 Thus we
 may assume that $\cX$ is affine.

By definition of the Berkovich topology on $X^{\an}$, it is enough
to prove that for every regular function $f$ on $X=\cX_K$ the map
$$|\Delta(\cX_k)|\to \R_{\geq 0},\, x\mapsto |f(\Sk_{\cX}(x))|$$ is
continuous. This follows at once from the fact that the
quasi-monomial valuation $v_{\alpha}$ in Proposition
\ref{prop-qmon} is continuous in the parameters $\alpha$.
\end{proof}

\sss \label{sss-rho}
 We can construct a continuous
retraction
$$\rho_{\cX}\colon \widehat{\cX}_\eta\to \Sk(\cX)$$ for the inclusion $\Sk(\cX)\to \widehat{\cX}_\eta$ as follows. Let
$x$ be a point of $\widehat{\cX}_\eta$, and denote by $\zeta$ its
image under the reduction map
$$\red_{\cX}\colon \widehat{\cX}_\eta\to \cX_k.$$ Let $E_1,\ldots,E_r$ be the
irreducible components of $\cX_k$ that pass through $\zeta$.
 Let $\xi$ be the generic
point of the connected component of $E_1\cap\ldots\cap E_r$  that
 contains $\zeta$, and set $$w_{E_i}= v_x(E_i)$$ for all $i$. Then
$\rho_{\cX}(x)$ is the point with barycentric representation
$(\xi,w)$, where we used the map $\Sk_{\cX}$ to identify
$|\Delta(\cX_k)|$ and $\Sk(\cX)$. In other words, $\rho_{\cX}(x)$
is the monomial point on $X^{\an}$ represented by
$(\cX,(E_1,\ldots,E_r),\xi)$ and $w=(w_{E_1},\ldots,w_{E_r})$.
This is the unique monomial point $y$ on $X^{\an}$ such that $\cX$
is adapted to $x$, $v_y(E)=v_x(E)$ for every prime component $E$
of $\cX_k$, and $\red_{\cX}(x)$ belongs to the closure of
$\red_{\cX}(y)$. It is straightforward to check that $\rho_{\cX}$
is continuous and right inverse to the inclusion $\Sk(\cX)\to
\widehat{\cX}_\eta$.

\begin{prop}\label{prop-monmin}
Let $\cX$ be an $sncd$-model for $X$ and let $x$ be a point of
$\widehat{\cX}_\eta$. Then
$$|f(x)|\leq |f(\rho_{\cX}(x))|$$ for every $f$ in
$\mathcal{O}_{\cX,\red_{\cX}(x)}$.
\end{prop}
\begin{proof}
 Let $E_1,\ldots,E_r$ be the irreducible components of
$\cX_k$ passing through $\red_{\cX}(x)$, and set
$\alpha_i=v_x(E_i)$ for all $i$. Denote by $\xi$ the generic point
of the connected component of $\cap_{i=1}^rE_i$ containing
$\red_{\cX}(x)$. Then $\rho_{\cX}(x)$ is the monomial point of
$X^{\an}$ represented by $(\cX,(E_1,\ldots,E_r),\xi)$ and
$\alpha$, and $\xi=\red_{\cX}(\rho_{\cX}(x))$.

 Let $z_i=0$ be a local equation for $E_i$ on $\cX$ at $\red_{\cX}(x)$, for all $i$. Then $(z_1,\ldots,z_r)$ is a regular system of local parameters in
 $\mathcal{O}_{\cX,\xi}$.
 By Lemma \ref{lemm-admiss}, one can always find an admissible expansion $$f=\sum_{\beta\in
\N^r}c_{\beta}z^{\beta}$$ for $f$ in
$\widehat{\mathcal{O}}_{\cX,\xi}$ such that all the coefficients
$c_{\beta}$ belong to $\mathcal{O}_{\cX,\red_{\cX}(x)}$.
 We have $|h(x)|\leq 1$ and $|h(\rho_{\cX}(x))|\leq 1$ for every
 function $h$ in $\mathcal{O}_{\cX,\red_{\cX}(x)}\subset
 \mathcal{O}_{\cX,\xi}$.
Since, moreover, $|z_i(x)|=|z_i(\rho_{\cX}(x))|<1$ for all $i$ in
$\{1,\ldots,r\}$,
 we can rewrite the expression for $f$ as
 $$f=\sum_{\beta\in
S}c_{\beta}z^{\beta}+g$$ where $S$ is a finite subset of $\N^r$
and $g$ is an element of $\mathcal{O}_{\cX,\red_{\cX}(x)}$ such
that $|g(x)|<|f(x)|$ and $|g(\rho_{\cX}(x))|<|f(\rho_{\cX}(x))|$.

Then we have
\begin{eqnarray*} |f(x)|&\leq&
\max\{|c_{\beta}(x)|\cdot |z^{\beta}(x)|\ |\,\beta\in S\}
\\ &\leq &\max\{|z^{\beta}(x)|\ |\,\beta\in S\}
\\ &=&|(f-g)(\rho_{\cX}(x))|
\\ &=&|f(\rho_{\cX}(x))|.
\end{eqnarray*}
\end{proof}

\begin{prop}\label{prop-domin}
If $h\colon \cX'\to \cX$ is an $R$-morphism of $sncd$-models of $X$,
then
 $\widehat{\cX'}_\eta\subset \widehat{\cX}_\eta$ and $$(\rho_{\cX}\circ \rho_{\cX'})(x)=\rho_{\cX}(x)$$ for every point $x$ in $\widehat{\cX'}_\eta$.
  If $h$ is proper, then $\widehat{\cX'}_\eta= \widehat{\cX}_\eta$
  and $\Sk(\cX)\subset \Sk(\cX')$.
\end{prop}
\begin{proof}
It is obvious that $\widehat{\cX'}_\eta\subset
\widehat{\cX}_\eta$. Let $x$ be a point of $\widehat{\cX'}_\eta$
and set $y=\rho_{\cX'}(x)$. Then $\red_{\cX'}(x)$ lies in the
closure of $\{\red_{\cX'}(y)\}$, which implies the analogous
statement for $\red_{\cX}(x)=h(\red_{\cX'}(x))$ and
$\red_{\cX}(y)=h(\red_{\cX'}(y))$. Thus by definition of the map
$\rho_{\cX}$, it is enough to show that $v_x(E)=v_{y}(E)$ for each
prime component $E$ of $\cX_k$. Writing
$$h^*E=\sum_{i=1}^r \alpha_i E'_i$$ where $E'_1,\ldots,E'_r$ are
prime components of $\cX'_k$, we compute:
$$v_x(E)=\sum_{i=1}^r \alpha_i v_x(E'_i)=\sum_{i=1}^r \alpha_i
v_y(E'_i)=v_y(E)$$ where the second equality follows from the
definition of the map $\rho_{\cX'}$.

Now assume that $h$ is proper. Then
 $\widehat{\cX'}_\eta= \widehat{\cX}_\eta$ by the valuative
 criterion for properness. Let $x$ be a point of
 $\Sk(\cX)$ and set $\xi=\red_{\cX}(x)$,   $x'=\rho_{\cX'}(x)$ and $\xi'=\red_{\cX'}(x')$. To prove that
 $\Sk(\cX)\subset \Sk(\cX')$, we must show that $x=x'$. It
 suffices to show that $v_x(f)=v_{x'}(f)$ for every element $f$ in
 $\mathcal{O}_{\cX,\xi}$.
 We choose a system of coordinates $(z_1,\ldots,z_r)$ and an admissible expansion $$f= \sum_{\beta\in
 \N^r}c_{\beta} z^{\beta}$$ for $f$ as in \eqref{sss-mondef}.
  Let $(z'_1,\ldots,z'_s)$ be a regular system of local parameters
  in $\mathcal{O}_{\cX',\xi'}$ such that there exist a unit $u'$
  in $\mathcal{O}_{\cX',\xi'}$ and positive integers
  $N'_1,\ldots,N'_s$ such that
$$u'\prod_{j=1}^s (z'_j)^{N'_j}$$ is a uniformizer in $R$.
 For each $i$ in $\{1,\ldots,r\}$, we can write
 $$z_i=v_i\prod_{j=1}^s (z'_j)^{\gamma_{ij}}$$ in
 $\mathcal{O}_{\cX',\xi'}$, with $v_i$ a unit and $\gamma_{ij}$
 non-negative integers.
We view
 $\gamma_i=(\gamma_{i1},\ldots,\gamma_{is})$ as a vector in
 $\N^s$, for each $i$.
Let $\beta$ and $\beta'$ be elements of $\N^r$ such that
$$\sum_{i=1}^r \beta_i\gamma_i= \sum_{i=1}^r
\beta'_i\gamma_i.$$ This means that the monomials $z^{\beta}$ and
$z^{\beta'}$ have the same multiplicity along each irreducible
component of $\cX'_k$, and thus that the rational function
$z^{\beta-\beta'}$ on $X$ is a unit in
$$\mathcal{O}_{\cX,\xi}\cong \mathcal{O}(\cX'\times_{\cX}\Spec
\mathcal{O}_{\cX,\xi})$$
 (the isomorphism follows from the fact that $h$ is proper, $\cX$ and $\cX'$
 are $R$-flat,
 $h_K$ is an isomorphism and $\cX$ is normal).
 This can only happen if $\beta=\beta'$.  It follows that
 $$f=\sum_{\beta\in
 \N^r}c_{\beta}v^{\beta} (z')^{\sum_{i=1}^r \beta_i\gamma_i}$$ is an
 admissible expansion for $f$ in
 $\widehat{\mathcal{O}}_{\cX',\xi'}$.
  Thus
 \begin{eqnarray*}
 v_{x'}(f)&=&\min\{\sum_{i,j}
 v_{x'}(z'_j)\beta_i\gamma_{ij}\,|\,\beta\in \N^r,\,c_\beta\neq 0\}
\\ &=&\min\{\sum_{i}
(\sum_j v_{x}(z'_j)\gamma_{ij})\beta_i\,|\,\beta\in
\N^r,\,c_\beta\neq 0\}
\\ &=&\min\{\sum_{i}
 v_{x}(z_i)\beta_i\,|\,\beta\in \N^r,\,c_\beta\neq 0\}
\\ &=&v_x(f).
\end{eqnarray*}
\end{proof}

\sss If $h\colon \cX'\to \cX$ is a proper morphism of $sncd$-models of
$X$, then in general, the skeleton $\Sk(\cX')$ will be strictly
larger than the skeleton $\Sk(\cX)$. The following proposition
shows, however, that the skeleton does not change if we blow up a
connected component of an intersection of irreducible components
of $\cX_k$.

\begin{prop}\label{prop-blup}
Let $\cX$ be an $sncd$-model for $X$, let $E_1,\ldots,E_r$ be
irreducible components of $\cX_k$, and let $\xi$ be a generic
point of $\cap_{i=1}^r E_i$. Denote by $h\colon \cX'\to \cX$ the blow-up
of $\cX$ at the closure of $\{\xi\}$. Then $\Sk(\cX')=\Sk(\cX)$.
More precisely, if we denote by $\sigma$ the face of $\Sk(\cX)$
corresponding to $\xi$, then the simplicial structure on
$\Sk(\cX')$ is obtained by adding a vertex at the barycenter of
$\sigma$, corresponding to the exceptional divisor of $h$, and
joining it to all the faces of $\sigma$.
\end{prop}
\begin{proof}
This follows from a straightforward computation.
\end{proof}

\subsection{Piecewise affine structure on the Berkovich
skeleton}\label{ss-pwaff}
 \sss Let $X$ be a connected regular
separated $K$-scheme of finite type, and let $\cX$ be an
$sncd$-model for $X$. We will use the simplicial structure of
$\Delta(\cX_k)$ to define an integral
 piecewise affine structure on the Berkovich skeleton $\Sk(\cX)$.
 We use the homeomorphism $$\Sk_{\cX}\colon |\Delta(\cX_k)|\to
 \Sk(\cX)$$ to identify $\Sk(\cX)$ with the geometric realization
 $|\Delta(\cX_k)|$ of $\Delta(\cX_k)$. Let $U$ be a subset of $\Sk(\cX)$ and consider a function
 $$f\colon U\to \R.$$ We say that $f$ is affine (with respect to the model $\cX$) if the following two conditions are fulfilled.
\begin{enumerate}
\item The function $f$ is continuous. \item For every closed face
$\sigma$ of $\Sk(\cX)$, we can cover $\sigma\cap U$ by open
subsets $V$ of $\sigma\cap U$ such that the restriction of $f$ to
$V$ is an affine
 function with coefficients in $\Z$ in the variables
$$(\frac{w_{E_1}}{N_1},\ldots,\frac{w_{E_r}}{N_r})$$ where
$E_1,\ldots,E_r$ are the irreducible components of $\cX_k$
corresponding to the vertices of $\sigma$, $N_1,\ldots,N_r$ are
their multiplicities in $\cX_k$, and $(w_{E_1},\ldots,w_{E_r})$
are the barycentric coordinates on $\sigma$.
\end{enumerate}
 We say that $f$ is piecewise affine (with respect to the model $\cX$) if $f$ is continuous and if
 we can cover each face $\sigma$ of $\Sk(\cX)$ by finitely
 many polytopes $P$ such that the vertices of $P$ have rational
 barycentric coordinates and the restriction of $f$ to $U\cap P$
 is affine.

\begin{prop}\label{prop-pieceaff1}
Let $\cX$ be an $sncd$-model of $X$, and let $h$ be a non-zero
rational function on $X$. Then the function
$$f_h\colon \Sk(\cX)\to \R,\, x\mapsto v_x(h)$$ is piecewise affine.
 If $\sigma$ is a closed
face of $\Sk(\cX)$ corresponding to a generic point $\xi$ of an
intersection of irreducible components of $\cX_k$, then
 the function $$f_h|_{\sigma}\colon \sigma\to \R,\,x\mapsto v_x(h)$$
 is
\begin{enumerate}
\item \label{item:concave}  concave if $\xi$ is not contained in
the closure of the locus of poles of $h$ on $X$,
 \item \label{item:convex}  convex if $\xi$ is not contained in the
closure of the locus of zeroes of $h$ on $X$,
 \item \label{item:affine} affine if $\xi$ is not contained in
the closure of the locus of zeroes and poles of $h$ on $X$.
\end{enumerate}
\end{prop}
\begin{proof}
Continuity of $f_h$ follows immediately from the definition of the
Berkovich topology. The remainder of the statement can be checked
on the restrictions of $f$ to the faces of the skeleton.  Let
$E_1,\ldots,E_r$ be irreducible components of $\cX_k$, let $\xi$
be a generic point of their intersection and let $\sigma$ be the
closed face of $\Sk(\cX)$ corresponding to $\xi$. We may assume
that $\xi$ does not lie in the closure of the locus of poles of
$h$ on $X$, because $f_{1/h}=-f_{h}$ so that \eqref{item:convex}
follows from \eqref{item:concave}. Multiplying $h$ by a suitable
element in $R$, we can further reduce to the case where $h$
belongs to $\mathcal{O}_{\cX,\xi}$.
 Taking an admissible expansion for $h$ in
 $\widehat{\mathcal{O}}_{\cX,\xi}$, we deduce from
 \eqref{eq-minimum}
 that $f_h|_{\sigma}$ is
 a minimum of finitely many affine functions on $\sigma$ and therefore a concave
 piecewise affine function. Point \eqref{item:affine} now
 follows from \eqref{item:concave} and \eqref{item:convex},
 because a piecewise affine concave and convex function is affine.
\end{proof}

\begin{prop}\label{prop-pieceaff2}
Let $\cX$ and $\cY$ be $sncd$-models of $X$. Then the map
$$\rho_{\cX}|_{\widehat{\cX}_\eta\cap\Sk(\cY)}\colon
\widehat{\cX}_\eta\cap \Sk(\cY)\to \Sk(\cX),\, y\mapsto \rho_{\cX}(y)$$ is
 piecewise affine, in the following sense: if $U$ is a subset of
$\Sk(\cX)$ and $f\colon U\to \R$ is a piecewise affine function on $U$,
then $f\circ \rho_{\cX}$ is a  piecewise affine function on
$\rho^{-1}_{\cX}(U)\cap \Sk(\cY)$.
\end{prop}
\begin{proof}
This is an easy consequence of Proposition \ref{prop-pieceaff1}
and the definition of the map $\rho_{\cX}$.
\end{proof}
As a corollary, we see that the notion of piecewise affine
function is intrinsic on $X$ and independent of the choice of an
$sncd$-model.
\begin{cor}\label{cor-pieceaff}
If $\cX$ and $\cY$ are $sncd$-models of $X$ and $U$ is a subset of
$\Sk(\cX)\cap \Sk(\cY)$, then a function $f\colon U\to \R$ is piecewise
affine with respect to $\cX$ if and only if it is piecewise affine
with respect to $\cY$.
\end{cor}
\begin{proof}
This follows immediately from Proposition \ref{prop-pieceaff2}.
\end{proof}
 Thus we get a canonical piecewise affine structure on the set of monomial points in $X^{\an}$. Beware, however, that the notion of affine
 function depends on the chosen model.

\section{The weight function and the Kontsevich-Soibelman
skeleton}\label{sec-weight}
\subsection{Some reminders on canonical sheaves}
\sss Let $f\colon Y\to X$ be a morphism of finite type of locally
Noetherian schemes, and assume that $f$ is a local complete
intersection. Then one can define the canonical sheaf
$\omega_{Y/X}$ of the morphism $f$; see Section 6.4.2 and Exercise
6.4.6 in \cite{liu}.
 It is a line bundle on $Y$, whose restriction to
the smooth locus of $f$ is canonically isomorphic to the sheaf of
relative differential forms of maximal degree of $f$. We will
mainly be interested in the situation where $Y$ and $X$ are
regular; then every morphism of finite type $Y\to X$ is a local
complete intersection, by \cite[6.3.18]{liu}. If $g\colon Z\to Y$ is
another local complete intersection morphism between locally
Noetherian schemes, then the canonical sheaves of $f$, $g$ and
$f\circ g$ are related by the adjunction formula: there is a
canonical isomorphism \begin{equation}\label{eq-adj}
\omega_{Z/X}\cong
\omega_{Z/Y}\otimes_{\mathcal{O}_Z}g^*\omega_{Y/X},\end{equation}
by Theorem 6.4.9 and Exercise 6.4.6 in \cite{liu}.

\sss If $X$ and $Y$ are integral and $f$ is dominant and smooth at
the generic point of $Y$, then the function field $F(Y)$ of $Y$ is
a separable extension of the function field $F(X)$ of $X$, of
finite transcendence degree $d$. If $\xi$ is the generic point of
$Y$, then we can view $\omega_{Y/X}$ as a
sub-$\mathcal{O}_{Y}$-module of the constant sheaf on $Y$
associated to the rank one $F(Y)$-vector space
$$(\omega_{Y/X})_{\xi}\cong \Omega^d_{F(Y)/F(X)}.$$

\sss\label{sss-compcansh} Let us recall how, in suitable
situations, the canonical sheaf $\omega_{Y/X}$ can be explicitly
computed. Assume that $X$ and $Y$ are regular integral schemes,
and let $f:Y\to X$ be a morphism of finite type such that $f$ is
dominant and smooth at the generic point of $Y$. Denote by $d$ the
transcendence degree of $F(Y)$ over $F(X)$, and let $y$ be a point
of $Y$. Locally at $y$, the morphism $f$ is of the form
$$
\Spec A[T_1,\ldots,T_n]/(F_1,\ldots,F_r)\to \Spec A$$ where
 the sequence $(F_1,\ldots,F_r)$ is regular at $y$ and $r\leq n$. Then
$n=r+d$. Renumbering the polynomials $F_i$, we may assume that
$$\Delta=\det\left( \left[ \frac{\partial F_i}{\partial T_j}\right]_{1\leq i,j\leq r}\right)$$
 is different from zero in $\mathcal{O}_{Y,y}$. Then by \cite[6.4.14]{liu}, the stalk $(\omega_{Y/X})_y$ of the canonical sheaf
 $\omega_{Y/X}$ at the point $y$
 is the sub-$\mathcal{O}_{Y,y}$-module of $\Omega^d_{F(Y)/F(X)}$
 generated by
$$\Delta^{-1}(dT_{r+1}\wedge \ldots\wedge dT_n).$$

\sss \label{sss-relcan} Now let $f\colon Y\to X$ be a dominant morphism
of finite type of regular integral schemes, and assume that $f$ is
\'etale over the generic point of $X$. This means that the
function field $F(Y)$ is a finite separable extension of $F(X)$.
Then we can view $\omega_{Y/X}$ as a sub-$\mathcal{O}_{Y}$-module
of the constant sheaf
$$\Omega^0_{F(Y)/F(X)}\cong F(Y)$$ on $Y$. This embedding defines
a canonical element in the linear equivalence of Cartier divisors
on $Y$ associated to the line bundle $\omega_{Y/X}$, namely, the
divisor of the element $1\in F(Y)$ viewed as a rational section of
$\omega_{Y/X}$. We call this divisor the relative canonical
divisor of $f$ and denote it by $K_{Y/X}$. It follows from
\eqref{sss-compcansh} that $K_{Y/X}$ is effective and that its
defining ideal sheaf $\mathcal{O}(-K_{Y/X})$ is precisely
$\mathrm{Fitt}^0 \Omega^1_{Y/X}$, the $0$-th Fitting ideal of
$\Omega^1_{Y/X}$. If $g:Z\to Y$ is a dominant morphism of finite
type of regular integral schemes such that $F(Z)$ is separable
over $F(Y)$ of transcendence degree $d$, then the isomorphism in
the adjunction formula \eqref{eq-adj} is an equality of
sub-$\mathcal{O}_Z$-modules of the constant sheaf
$$\Omega^d_{F(Z)/F(Y)}\cong \Omega^d_{F(Z)/F(X)}$$ on $Z$.

\begin{example}\label{ex-blup}
Let $X$ be a regular integral scheme, and let $Z$ be a regular
integral closed subscheme of codimension $r$. If $f\colon Y\to X$ is the
blow-up of $X$ at $Z$, then one can use \eqref{sss-compcansh} to
compute that the relative canonical  divisor of $f$ is equal to
$(r-1)E$, with $E=f^{-1}(Z)$ the exceptional divisor of $f$.
\end{example}

\begin{prop}\label{prop-relcan}
Let $\cX$ and $\cY$  be regular flat $R$-schemes of finite type,
and assume that $\cY_k$ and $\cX_k$ are irreducible and
$(\cX_k)_{\red}$ is regular. Let $h\colon \cY\to \cX$ be a dominant
$R$-morphism such that $h$ is \'etale over the generic point of
$\cX$.
 We denote by $M$ and $N$ the multiplicities of $\cX_k$ and
 $\cY_k$, respectively, and by $\nu-1$ the multiplicity of $(\cY_k)_{\red}$ in the
 relative canonical divisor $K_{\cY/\cX}$. Then $M$ divides $N$ and $\nu\geq N/M$.
\end{prop}
\begin{proof}
 First, we prove that $M$ divides $N$. Let $y$ be the generic point of $\cY_k$ and set $x=h(y)$. Then we
can find a prime element $f$ and a unit $u$ in
$\mathcal{O}_{\cX,x}$ such that $\pi=f^Mu$ is a uniformizer in
$R$. Thus the rational function $h^*f^M$ on $\cY$ has order $N$
along $\cY_k$, which implies that $M$ divides $N$.

 Now, we prove that $\nu\geq M/N$.   Shrinking $\cX$ and $\cY$, we
may assume that the following properties hold:
\begin{itemize}
\item $\cX$ and $\cY$ are affine, \item  the element $f\in
\mathcal{O}_{\cX,x}$ is defined at every point of $\cX$, \item
there exist an integer $m>0$, a polynomial $P$ in
$\mathcal{O}(\cX)[T_1,\ldots,T_m]$ and a surjective morphism of
$\mathcal{O}(\cX)$-algebras
$$\varphi\colon A:=\mathcal{O}(\cX)[T_1,\ldots,T_m]/(PT^{N/M}_1-f)\to
\mathcal{O}(\cY).$$
\end{itemize} Clearly, $\Spec A$ is regular at every point
lying above $x$, so that the morphism $\cY\to \Spec A$ is a
regular immersion at the point $y$.
 We set $F_1=PT_1^{N/M}-f$. Further shrinking $\cY$ around
$y$, we may assume that there exist an integer $n\geq m$ and
polynomials $F_2,\ldots,F_n$ in $\mathcal{O}(\cX)[T_1,\ldots,T_n]$
such that the sequence $(F_1,\ldots,F_n)$ is a regular sequence in
$\mathcal{O}(\cX)[T_1,\ldots,T_n]$ and such that the morphism
$\varphi$ factors through an isomorphism of
$\mathcal{O}(\cX)$-algebras
$$\mathcal{O}(\cX)[T_1,\ldots,T_n]/(PT^{N/M}_1-f,F_2,\ldots,F_n)\to
\mathcal{O}(\cY).$$ Applying \eqref{sss-compcansh}, we now compute
that $\nu\geq N/M$.
\end{proof}

\subsection{The weight of a differential form at a divisorial valuation}
\sss Let $X$ be a connected smooth separated $K$-scheme of
dimension $n$ and let $\omega$ be a non-zero rational
$m$-pluricanonical form on $X$, for some $m>0$. Thus $\omega$ is a
non-zero rational section of the $m$-pluricanonical line bundle
$\omega_{X/K}^{\otimes m}$ on $X$. In the following subsections,
we will explain how one can use $\omega$
 to define a weight function on the analytic space
 $X^{\an}$. As we will see, if $X$ is proper and $\omega$ is regular at every point
 of $X$, then the weight function has the interesting property that it
 is strictly increasing if one moves away from the Berkovich
 skeleton associated to any proper $sncd$-model of $X$. We will
 use this property to compute the so-called Kontsevich-Soibelman
 skeleton of the pair $(X,\omega)$ in Theorem \ref{thm-KS}.

\sss For each regular $R$-model $\cX$ of $X$, the pluricanonical
 form $\omega$ defines a rational section of the relative $m$-pluricanonical
line bundle $\omega^{\otimes m}_{\cX/R}$ and thus a divisor
$\mathrm{div}_{\cX}(\omega)$ on $\cX$. If $h\colon \cX'\to \cX$ is
a morphism of regular $R$-models of $X$, then it follows from
\eqref{sss-relcan} that
$$\mathrm{div}_{\cX'}(\omega)= h^*\mathrm{div}_{\cX}(\omega)+mK_{\cX'/\cX}.$$

\sss \label{sss-wdiv} We first define the weight function on
divisorial points. Let $x$ be a divisorial point of $X^{\an}$,
associated to a regular $R$-model $\cX$ of $X$ and an irreducible
component $E$ of $\cX_k$.   We denote by $\mu$ the unique integer
 such that $\mu-m$ equals the multiplicity of $E$ in
$\mathrm{div}_{\cX}(\omega)$, and by $N$ the multiplicity of $E$
in $\cX_k$.

\begin{prop}
The rational number $\mu/N$ only depends on $X$, $\omega$ and $x$,
and not on the choice of the model $\cX$.
\end{prop}
\begin{proof}
 Let $\cY$ be another regular $R$-model of $X$ and let $F$ be an
 irreducible component of $\cY_k$ such that $x$ is the divisorial
 point associated to $(\cY,F)$.
Then the isomorphism between the generic fibers of $\cX$ and $\cY$
will extend to an isomorphism between some open neighbourhoods of
the generic points of $E$ and $F$, respectively, because the local
rings at these points are the same when viewed as subrings of the
field of rational functions on $X$: they both coincide with the
valuation ring of $v_x$. Thus the value $\mu/N$ does not depend on
the choice of the $R$-model $\cX$.
\end{proof}

\sss \label{sss-divweight}  We call $\mu/N$ the weight of $\omega$
at the point $x$, and denote it by $\weight_{\omega}(x)$. In this
way, we obtain a function
$$\weight_{\omega}\colon \mathrm{Div}(X)\to \Q,\, x\mapsto
\weight_{\omega}(x)$$ on the set of divisorial points of
$X^{\an}$, which we call the weight function associated to
$\omega$. Note that
$$\weight_{\omega^{\otimes d}}(x)=d\cdot\weight_{\omega}(x)$$ for every
integer $d>0$ and
$$\weight_{f\omega}(x)=\weight_{\omega}(x)+v_x(f)$$ for every
non-zero rational function $f$ on $X$.

\sss \label{sss-mu} In the following subsections, it will be
convenient to have a formula for the weight of $\omega$ at a
divisorial point in terms of a monomial presentation. Let $x$ be a
monomial point on $X^{\an}$, represented by an $sncd$-triple
$(\cX,(E_1,\ldots,E_r),\xi)$ and a tuple $\alpha$ in $\R^r_{> 0}$.
 Replacing $\cX$ by its open subscheme of regular points does not
influence the associated monomial valuation, so that we may assume
that $\cX$ is regular. As before, the rational $m$-pluricanonical
form $\omega$ defines a rational section of the relative
$m$-pluricanonical sheaf $\omega^{\otimes m}_{\cX/R}$, and thus a
divisor $\mathrm{div}_{\cX}(\omega)$ on $\cX$.  For each $i$ in
$\{1,\ldots,r\}$, we denote by $\mu_i$ the unique integer such
that $\mu_i-m$ equals the multiplicity of
$\mathrm{div}_{\cX}(\omega)$ along $E_i$. Recall that we
associated a value $v_x(E)$ to every divisor $E$ on $\cX$ in
\eqref{sss-model}.


\begin{lemma}\label{lemm-divindep}
If $x$ is divisorial, then
$$\weight_{\omega}(x)=v_x(\mathrm{div}_{\cX}(\omega)+m(\cX_k)_{\red}).$$
In particular, if $\xi$ is not contained in the closure of the
locus of zeroes and poles of $\omega$ on $X$, then
$$\weight_{\omega}(x)=\sum_{i=1}^r\alpha_i\mu_i.$$
\end{lemma}
\begin{proof}
The second assertion follows immediately from the first, since in
this case, locally around $\xi$, we have
$$\mathrm{div}_{\cX}(\omega)=\sum_{i=1}^r(\mu_i-m)E_i$$ so that
$$v_x(\mathrm{div}_{\cX}(\omega))=\sum_{i=1}^r(\mu_i-m)\alpha_i.$$
 Thus it suffices to prove the first assertion. If $r=1$ this is simply the definition of the weight of $\omega$ at $x$, and we will reduce to this situation.
 We construct an $sncd$-triple $$(\cX',(E'_1,\ldots,E'_r),\xi')$$ and a tuple $\alpha'$ in $\Q^r_{>0}$ as in the proof of Proposition
 \ref{prop-rank1}, by blowing up $\cX$ at the unique connected component of
 $\cap_{i=1}^r E_i$ that contains $\xi$. As we have explained in the
 proof of Proposition
 \ref{prop-rank1}, repeating this operation a finite number of
 times, we can reduce to the case where $r=1$.
 Thus it is enough to show that the value
 $$v_x(\mathrm{div}_{\cX}(\omega)+m(\cX_k)_{\red})$$ does not
 change under such blow-ups, i.e.,
 $$v_x(\mathrm{div}_{\cX}(\omega))+m\sum_{i=1}^r\alpha_i=v_x(\mathrm{div}_{\cX'}(\omega))+m\sum_{i=1}^r\alpha'_i.$$

  The relative canonical divisor of the
 blow-up morphism $\cX'\to \cX$ is equal to $(r-1)E'_1$, so that
 $$h^*\mathrm{div}_{\cX}(\omega)=\mathrm{div}_{\cX'}(\omega)-m(r-1)E'_1.$$  Since $\alpha'_1=\alpha_1$ and $\alpha'_i=\alpha_i-\alpha_1$ for $2\leq i\leq r$, we obtain:
 \begin{eqnarray*}
v_x(\mathrm{div}_{\cX'}(\omega))+m\sum_{i=1}^r\alpha'_i
 &=&
v_x(h^*\mathrm{div}_{\cX}(\omega))+m(r-1)\alpha_1+m\sum_{i=1}^r\alpha'_i
\\ &=&v_x(\mathrm{div}_{\cX}(\omega))+m\sum_{i=1}^r\alpha_i.
 \end{eqnarray*}
\end{proof}

\begin{lemma}\label{lemm-ineq}
Let $\cY$ be an $sncd$-model for $X$ and let $y$ be a divisorial
point on $\widehat{\cY}_\eta$. Then $$ \weight_{\omega}(y)\geq
v_y(\mathrm{div}_{\cY}(\omega)+m(\cY_k)_{\red})$$ and equality
holds if and only if $y$ lies on the Berkovich skeleton
$\Sk(\cY)$.
\end{lemma}
\begin{proof}
Denote by $E_1,\ldots,E_r$ the irreducible components of $\cY_k$
that contain $\red_{\cY}(y)$, and let $\xi$ be the generic point
of the connected component of $\cap_{i=1}^rE_i$ containing
$\red_{\cY}(y)$. We set $y'=\rho_{\cY}(y)$.  The point $y'$ is
divisorial because the value group of $v_{y'}$ is contained in the
value group of $v_y$. Multiplying $\omega$ with a suitable
rational function on $X$, we may assume that $\red_{\cY}(y)$ is
not contained in the closure of the locus of zeroes and poles of
$\omega$ on $X$.
 Then
$$v_y(\mathrm{div}_{\cY}(\omega)+m(\cY_k)_{\red})=v_{y'}(\mathrm{div}_{\cY}(\omega)+m(\cY_k)_{\red})=\weight_{\omega}(y')$$
 by Lemma \ref{lemm-divindep}, so that it is enough to
show that
$$\weight_{\omega}(y)\geq \weight_{\omega}(y'),$$ with equality if
and only if $y=y'$.


By means of the blowing up procedure explained in the proof of
Proposition \ref{prop-rank1}, we construct an $sncd$-model $\cX$
of $X$ and an irreducible component $E$ of $\cX_k$ such that $y'$
is the divisorial point associated to $(\cX,E)$.
 Then Propositions \ref{prop-blup} and \ref{prop-domin} imply that
 $y'=\rho_{\cX}(y)$. Thus we can assume without loss of generality
that $r=1$. We set $E=E_1$.

 Since every element of $\mathcal{O}_{\cY,\red_{\cY}(y)}$ belongs to the
valuation ring of $v_y$, we can find a regular $R$-model $\cZ$ of
$X$ such that $\cZ_k$
 is irreducible and $y$ is the divisorial point associated to
 $(\cZ,(\cZ_k)_{\red})$, together with a morphism $h\colon \cZ\to \cY$ of
 $R$-models of $X$. This morphism sends the generic point of $\cZ_k$ to
 $\xi$. We set $F=(\cZ_k)_{\red}$.

 We denote by $\mu-m$ the
multiplicity of $E$ in $\mathrm{div}_{\cY}(\omega)$, by $M$ the
multiplicity of $E$ in $\cY_k$, by $\nu-1$ the multiplicity of $F$
in $K_{\cZ/\cY}$ and by $N$ the multiplicity of $F$ in $\cZ_k$.
 If
  $f=0$ is a local equation for $E$ in $\cY$ at $\xi$, then
  $$v_y(f)=v_{y'}(f)=\frac{1}{M}$$ because, up to an invertible
  factor, $f^{M}$ is a uniformizer in $R$. This implies that $N$
  is a multiple of $M$ and that the multiplicity of $F$ in  $\mathrm{div}_{\cZ}(\omega)$  is equal to
  $$(\mu-m)\frac{N}{M}+m\nu-m.$$
  Now we compute:
  $$\weight_{\omega}(y)=\frac{1}{N}((\mu-m)\frac{N}{M}+m\nu-m+m)=\frac{\mu}{M}+\frac{m(\nu-N/M)}{N}.$$
On the other hand,
  $$\weight_{\omega}(y')= \frac{\mu}{M}.$$
It follows from Proposition \ref{prop-relcan} that $\nu\geq N/M$,
so that $$\weight_{\omega}(y)\geq \weight_{\omega}(y').$$ Now
assume that $y\neq y'$.  We denote by $Y$ the closure of
$\{\red_{\cY}(y)\}$ in $\cY$, endowed with its reduced induced
structure.
 Shrinking $\cY$, we may assume that $Y$ is regular. Note that $Y$ is a strict subvariety of $(\cY_k)_{\red}$, since otherwise,
  $y$ would be equal to $y'$.
   We denote by $g:\cY'\to \cY$ the blow-up of $\cY$ at $Y$.
 Then
\begin{eqnarray*}
v_y(\mathrm{div}_{\cY}(\omega)+m(\cY_k)_{\red})&=&v_y(g^*\mathrm{div}_{\cY}(\omega)+mg^*(\cY_k)_{\red})
\\ &=& v_y(\mathrm{div}_{\cY'}(\omega)-mK_{\cY'/\cY}+m(\cY'_k)_{\red})
\\ &<& v_y(\mathrm{div}_{\cY'}(\omega)+m(\cY'_k)_{\red}).
\end{eqnarray*}
 Applying the first part of the proof to the model $\cY'$, we get
$$v_y(\mathrm{div}_{\cY'}(\omega)+m(\cY'_k)_{\red})\leq
\weight_{\omega}(y)$$ and therefore
$$\weight_{\omega}(y')<\weight_{\omega}(y).$$
\end{proof}

\subsection{The weight function on a Berkovich skeleton}\label{sec-sncd}
\sss Let $X$ be a connected smooth proper $K$-scheme of dimension
$n$ and let $\omega$ be a non-zero rational $m$-pluricanonical
form on $X$.
 We suppose that $X$ has a proper
$sncd$-model $\cX$. Note that, at this point, we are not assuming
resolution of singularities,
 nor that any two proper $sncd$-models can be dominated by a common
one.
 We
denote by $E_i,\,i\in I$ the irreducible components of $\cX_k$
 and by $m\mu_i-m$ the multiplicity of $E_i$ in
$\mathrm{div}_{\cX}(\omega)$, for each $i\in I$.

\sss \label{sss-monweight} Let $x$ be a point of the Berkovich
skeleton $\Sk(\cX)$. Then we set
$$\weight_{\cX,\omega}(x)=v_x(\mathrm{div}(\omega)+m(\cX_k)_{\red}).$$ In this
way, we obtain a function
$$\weight_{\cX,\omega}\colon \Sk(\cX)\to \R$$ that we call the weight
function associated to $\cX$ and $\omega$.  It is clear from the
definition that, for each point $x$ of $\Sk(\cX)$, we have
$$\weight_{\cX,\omega^{\otimes d}}(x)=d\cdot\weight_{\cX,\omega}(x)$$
for all $d>0$ and
$$\weight_{\cX,f\omega}(x)=\weight_{\cX,\omega}(x)+v_x(f)$$ for
every non-zero rational function $f$ on $X$.

\begin{prop}\label{prop-monweight0}
The function $$\weight_{\cX,\omega}\colon \Sk(\cX)\to \R$$ is
piecewise affine.
 If $J$ is a non-empty subset of $I$, $\xi$ is a
generic point of $\cap_{j\in J}E_J$ and $\sigma$ is the closed
face of $\Sk(\cX)$ corresponding to $\xi$, then the restriction of
$\weight_{\cX,\omega}(x)$ to $\sigma$ is
\begin{itemize}
\item concave if $\xi$ is not contained in the closure of the
locus of poles of $\omega$ on $X$. \item convex if $\xi$ is not
contained in the closure of the locus of zeroes
 of $\omega$ on $X$.
 \item affine if $\xi$ is not contained in the closure of the locus of zeroes
and poles of $\omega$ on $X$.
\end{itemize}
\end{prop}
\begin{proof}
The function $\weight_{\cX,\omega}$ is continuous, by Proposition
\ref{prop-cont}. The remaining properties can be checked on the
restriction of $\weight_{\cX,\omega}$ to a closed face $\sigma$ as
in the statement. We can write $\omega$ as $h\cdot \omega_0$ where
$h$ is a non-zero rational function on $X$ and $\omega_0$ is a
generator of the stalk of $\omega^{\otimes}_{\cX/R}$ at $\xi$. If
we choose, for each $j$ in $J$, a local equation $x_j=0$ for the
prime divisor $E_j$ on $\cX$ at $\xi$, then we have
$$\weight_{\cX,\omega}(x)=v_x(h\prod_{j\in J}x^m_j)$$ for every
$x$ in $\sigma$, so that the result follows from Proposition
\ref{prop-pieceaff1}.
\end{proof}

\begin{prop}\label{prop-monweight}\item
\begin{enumerate}
\item \label{item:div} If $x$ is a divisorial point on $\Sk(\cX)$,
then
$$\weight_{\cX,\omega}(x)=\weight_{\omega}(x).$$
 \item
\label{item:incr} If $x$ is a divisorial point on $X^{\an}$ such
that $\red_{\cX}(x)$ is not contained in the closure of the locus
of poles of $\omega$ on $X$, then
$$\weight_{\omega}(x)\geq \weight_{\cX,\omega}(\rho_{\cX}(x))$$
with equality if and only if $x\in \Sk(\cX)$.
 \item
\label{item:incr2} If $\cY$ is a proper $sncd$-model of $X$ and
$y$ is a point of $\Sk(\cY)$ such that $\red_{\cX}(y)$ is not
contained in the closure of the locus of poles of $\omega$ on $X$,
then
 $$\weight_{\cY,\omega}(y)\geq
\weight_{\cX,\omega}(\rho_{\cX}(y))$$  and equality can only occur
if $y\in \Sk(\cX)$.
 \item \label{item:indep} If $\cY$ is a proper $sncd$-model of $X$ and if $x$ is a point in
$\Sk(\cX)\cap \Sk(\cY)$, then
$$\weight_{\cX,\omega}(x)=\weight_{\cY,\omega}(x).$$
\end{enumerate}
\end{prop}
\begin{proof}
\eqref{item:div} This follows from Lemma \ref{lemm-divindep}.

\eqref{item:incr} We may assume that $\red_{\cX}(x)$ is not
contained in the closure of the zero locus
 of $\omega$ on $X$, by dividing $\omega$ by a suitable element of
 $\mathcal{O}_{\cX,\zeta}$ and invoking Proposition
 \ref{prop-monmin}.
 Then it follows from Lemma \ref{lemm-ineq} that
$$\weight_{\omega}(x)\geq v_x(\sum_{i\in
I}(\mu_i-m)E_i+m(\cX_k)_{\red})
=\weight_{\cX,\omega}(\rho_{\cX}(x))$$ with equality if and only
if $x$ belongs to $\Sk(\cX)$.

\eqref{item:incr2} Denote by $P$ the set of points of $\cX_k$ that
belong to the closure of the locus of poles of $\omega$ on $X$.
This is a closed subset of $\cX_k$. Let $Z$ be the set of points
$y'$ on $\Sk(\cY)$ such that $\red_{\cX}(y')$ does not lie in $P$.
This is a closed subset of $\Sk(\cY)$ containing $y$, by
anti-continuity of the reduction map $\red_{\cX}$. We claim that
the divisorial points in $Z$ are dense in $Z$. Accepting this
claim for now, it suffices to prove the inequality in the
statement in the case where $y$ is divisorial, by continuity of
$\weight_{\cX,\omega}$, $\weight_{\cY,\omega}$  and $\rho_{\cX}$.
But when $y$ is divisorial the inequality follows from
\eqref{item:div} and \eqref{item:incr}. Likewise, if we denote by
$Z'$ the locus of points $y'$ in $Z$ such that
$$\weight_{\cY,\omega}(y')=\weight_{\cX,\omega}(\rho_{\cX}(y')),$$ then the divisorial points in $Z'$ form
 a dense subset, because the functions $\weight_{\cY,\omega}$ and $\weight_{\cX,\omega}\circ
 \rho_{\cX}$ are piecewise affine on $\Sk(\cY)$, by Propositions
 \ref{prop-pieceaff2} and \ref{prop-monweight0}. Point \eqref{item:incr} implies that
 $\mathrm{Div}(X)\cap Z'$ is contained in $\Sk(\cX)$, so that
 $Z'$ is contained in $\Sk(\cX)$.

 It remains to prove our claim that the set of divisorial points in $Z$ is dense in
 $Z$. Since $y$ is an arbitrary point of $Z$, it is enough to show that $y$ belongs to the closure of the set of divisorial points in $Z$.
 Denote by $\xi$ the center of $v_y$ on $\cX$ and by $C$ the closure of $\{\xi\}$ in $\cX$, endowed with
 its reduced induced structure. We choose finitely many affine open
neighbourhoods $\cU_1,\ldots,\cU_r$ of $\xi$ in $\cX$ such that
$C$ is contained in their union. For each $i$ in $\{1,\ldots,r\}$,
we choose a regular function $f_i$ in $\mathcal{O}(\cU_i)$ such
that the support of the closed subscheme of $\cU_i$ defined by
$f_i=0$ is equal to $\cU_i\cap P$. Then the set
$$S=\{z\in \red_{\cX}^{-1}(C)\,|\,v_z(f_i)=0\mbox{ for }i=1,\ldots,r\}$$ is
a locally closed subset of $Z$ containing $y$. The divisorial
points in $S$ form a dense subset of $S$, because
$\red_{\cX}^{-1}(C)$ is open in $\Sk(\cY)$ and the functions
$$\Sk(\cY)\to \R,\, z\to v_z(f_i)$$ are piecewise affine. It follows
that $y$ belongs to the closure in $Z$ of the set of divisorial
points in $Z$.

 \eqref{item:indep}
 Since the statement is symmetric in $\cX$ and $\cY$, it is enough
 to prove that
 $$\weight_{\cY,\omega}(x)\geq \weight_{\cX,\omega}(x).$$
 Multiplying $\omega$ by a suitable non-zero
 rational function on $X$, we can reduce to the case where the center of $v_x$ on
$\cX$ is not contained in the closure of the locus of poles of
$\omega$ on $X$. In that case, the result follows from
\eqref{item:incr2}.
\end{proof}

\subsection{The weight function on $X^{\an}$}

\sss\label{sss-assu} Let $X$ be a connected smooth $K$-scheme of
dimension $n$.
%
 We will show how weight functions can
be defined on the whole space  $X^{\an}$.
 Throughout this section, we make the following assumption: there
exists a smooth compactification $\overline{X}$ such that, for
every proper $R$-model $\overline{\cX}$ of $\overline{X}$, there
exists a morphism of $R$-models $h:\overline{\cX}'\to
\overline{\cX}$ such that $\overline{\cX}'$ is a proper
$sncd$-model and $h$ is an isomorphism over the open subscheme of
$\overline{\cX}$ consisting of the points where $\overline{\cX}$
is regular and $\overline{\cX}_k$ is a divisor with strict normal
crossings. This assumption is satisfied when $k$ has
characteristic zero or $X$ is a curve.

\begin{prop}\label{prop-assu} Let $x$ be a point of $X^{\an}$.
Every proper $R$-model $\overline{\cY}$ of $\overline{X}$ can be
dominated by a proper $sncd$-model $\overline{\cX}$ that has an
open subscheme $\cX$ such that $\cX$ is an $sncd$-model for $X$
and
 $\widehat{\cX}_\eta$ is a neighbourhood of
$x$ in $X^{\an}$. If $x$ is monomial, we can moreover arrange that
$x\in \Sk(\cX)$.
\end{prop}
\begin{proof}
If $x$ is monomial, we can first choose an $sncd$-model $\cY'$ of
$X$ such that $x$ lies in
 $\mathrm{Sk}(\cY)$ and compactify it to an $R$-model $\overline{\cY}'$ of
 $\overline{X}$. Replacing $\overline{\cY}$ by a proper $R$-model
 of $\overline{X}$ that dominates both $\overline{\cY}$ and
 $\overline{\cY}'$, we can assume that $x$ lies in
 $\Sk(\overline{\cY})$, by Proposition \ref{prop-domin}.

We set $Z=\overline{X}\smallsetminus X$ and we endow this closed subset
of $\overline{X}$ with its reduced induced structure. By Corollary
\ref{cor-ZR}, we can dominate $\overline{\cY}$ by a proper
$sncd$-model $\overline{\cX}$ of $\overline{X}$ such that the
closure $C$ of $\red_{\overline{\cX}}(x)$ in $\overline{\cX_k}$ is
disjoint from the closure of $Z$ in $\overline{\cX}$. The set
$\red^{-1}_{\overline{\cX}}(C)$ is open in $X^{\an}$ by
anti-continuity of the reduction map. Removing from
$\overline{\cX}$ the closure of $Z$, we get an open subscheme
$\cX$ which is an $sncd$-model of $X$ and such that
 $\widehat{\cX}_\eta$ is a neighbourhood of $x$ in $X^{\an}$.
 If $x$ lies in $\Sk(\overline{\cY})$, then it will also lie in
 $\Sk(\cX)$, by Proposition \ref{prop-domin}.
  \end{proof}

\begin{prop}\label{prop-wmon}
Let $\omega$ be a non-zero rational $m$-pluricanonical form on
$X$. There exists a unique function
$$\weight_{\omega}\colon \mathrm{Mon}(X)\to \Q$$ on the set of
monomial points of $X^{\an}$ such that
$$\weight_{\omega}(x)=v_x(\mathrm{div}_{\cX}(\omega)+m(\cX_k)_{\red})$$
for every monomial point $x$ of $X$ and every $sncd$-model $\cX$
of $X$ for which $x\in \Sk(\cX)$.
\end{prop}
\begin{proof}
This follows immediately from Proposition \ref{prop-monweight},
since we can compactify $\cX$ to a proper $sncd$-model of
$\overline{X}$.
\end{proof}

\sss \label{sss-weightext} Let $\omega$ be a non-zero regular
 $m$-pluricanonical form on $X$, for some $m>0$, and let $x$ be a point
of $X^{\an}$. We set
$$\weight_{\omega}(x)=\sup_{\cX}\{\weight_{\omega}(\rho_{\cX}(x))\}\in
\R\cup\{+\infty\}$$ where the supremum is taken over all
$sncd$-models $\cX$ of $X$ such that $x$ lies in
$\widehat{\cX}_\eta$. In this way, we obtain a function
$$\weight_{\omega}\colon X^{\an}\to \R\cup \{+\infty\}.$$

\begin{prop}\label{prop-weightext}\item
\begin{enumerate}
\item \label{item:semico} The function
$$\weight_{\omega}\colon X^{\an}\to \R\cup \{+\infty\}$$ is
lower semi-continuous. \item \label{item:compat} We have
$$\weight_{\omega}(x)=v_x(\mathrm{div}_{\cX}(\omega)+m(\cX_k)_{\red})$$
 for every  monomial point $x$ on $X^{\an}$ and every $sncd$-model $\cX$ of $X$  such that $x\in \Sk(\cX)$. In particular, the definition in \eqref{sss-weightext}
 agrees with the one in \eqref{sss-divweight} for divisorial
 points and with the one in Proposition \ref{prop-wmon} for points on the skeleton of an $sncd$-model of $X$.
 \item \label{item:supeq} For every point $x$ of $X^{\an}$ and every
$sncd$-model $\cX$ of $X$ such that $x\in \widehat{\cX}_\eta$, we
have
$$\weight_{\omega}(x)\geq \weight_{\omega}(\rho_{\cX}(x))$$ and
equality holds if and only if $x$ lies in $\Sk(\cX)$.

\item \label{item:birat} If $h:Y\to X$ is an open immersion, then
$$\weight_{h^*\omega}(y)=\weight_{\omega}(h(y))$$ for every point
$y$ on $Y^{\an}$.

\item \label{item:pluri} For each point $x$ in $X^{\an}$, we have
$$\weight_{\omega^{\otimes d}}(x)=d\cdot \weight_{\omega}(x)$$ for
all integers $d>0$ and
$$\weight_{f\omega}(x)=\weight_{\omega}(x)+v_x(f)$$ for all
 regular functions $f\neq 0$ on $X$.
\end{enumerate}
\end{prop}
\begin{proof}
\eqref{item:semico} Let $x$ be a point of $X^{\an}$. By
Proposition \ref{prop-assu}, we can find an $sncd$-model $\cX$ of
$X$ such that $\widehat{\cX}_\eta$ is a neighbourhood of $x$. For
every point $y$ of $\widehat{\cX}_\eta$, we have
$$\weight_{\omega}(x)=\sup_{\cX'}\{\weight_{\omega}(\rho_{\cX'}(x))\}\in
\R\cup\{+\infty\}$$ where the supremum is taken over all
 proper $R$-morphisms $\cX'\to \cX$ of
$sncd$-models of $X$. Thus $\weight_{\omega}$ is lower
semi-continuous on $\widehat{X}_\eta$, since it is the supremum of
 the continuous functions $\weight_{\omega}\circ \rho_{\cX'}$.

\eqref{item:compat} This follows from Proposition
\ref{prop-monweight}.

\eqref{item:supeq} The inequality follows immediately from the
definition. Proposition \ref{prop-monweight} implies that every
 monomial point $y$  of $\widehat{\cX}_\eta$ that does not lie on $\Sk(\cX)$
 satisfies
$$\weight_{\omega}(y)> \weight_{\omega}(\rho_{\cX}(y)).$$ Thus
it suffices to prove the following property: if $x$ does not lie
on $\Sk(\cX)$, then we can find a proper morphism $\cX'\to \cX$ of
$sncd$-models of $X$ such that $x':=\rho_{\cX'}(x)$ does not lie
on $\Sk(\cX)$. Note that $\rho_{\cX}(x')=\rho_{\cX}(x)$ by
Proposition \ref{prop-domin}.

 We choose a proper morphism of $sncd$-models $\cY\to \cX$
 such that
$$\red_{\cY}(x)\neq \red_{\cY}(\rho_{\cX}(x)).$$ The existence of such a model follows from Proposition \ref{prop-ZR}.
    We set $y=\rho_{\cY}(x)$.  If $y$ does not belong to $\Sk(\cX)$, then we can take $\cX'=\cY$ and $x'=y$. Thus we only need to consider the case where $y\in \Sk(\cX)$.
 Then we
  have $y=\rho_{\cX}(x)$
  so that $$\red_{\cY}(x)\neq \red_{\cY}(y)$$
and $\red_{\cY}(x)$ belongs to the closure of $\{\red_{\cY}(y)\}$.
It follows that $\red_{\cY}(y)$ cannot belong to the closure of
$\{\red_{\cY}(x)\}$. Blowing up $\cY$ at the closure of
$\{\red_{\cY}(x)\}$ and resolving the singularities, we obtain a
proper morphism of $sncd$-models $\cX'\to \cY$. We set
$x'=\rho_{\cX'}(x)$. Note that $\rho_{\cY}(x')=y$ by Proposition
\ref{prop-domin}. The inverse image in $\cX'_k$ of the closure of
$\{\red_{\cY}(x)\}$ is a union of irreducible components of
$\cX'_k$. This implies that $\red_{\cY}(x')$ lies in the closure
of $\{\red_{\cY}(x)\}$. In particular, $\red_{\cY}(x')$ is
different from $\red_{\cY}(y)$, which means that $x'$ cannot lie
in $\Sk(\cY)$. Therefore, it doesn't lie in $\Sk(\cX)$ either,
since $\Sk(\cX)$ is contained in $\Sk(\cY)$ by Proposition
\ref{prop-domin}.

\eqref{item:birat}   Let $\cY$ be an $sncd$-model of $Y$ such that
$y\in \widehat{\cY}_\eta$.
  Gluing $\cY_K$ to $X$ by means of the open immersion $h$, we get
  an $sncd$-model $\cX$ of $X$ such that $h(y)\in
  \widehat{\cX}_\eta$.
 Then it is clear that
$$\weight_{h^*\omega}(\rho_{\cY}(y))=\weight_{\omega}(\rho_{\cX}(h(y))\leq \weight_{\omega}(h(y)).$$
  Since this holds for all $sncd$-models $\cY$ of $Y$ with $y\in \widehat{\cY}_\eta$, we find
that $$\weight_{h^*\omega}(y)\leq \weight_{\omega}(x).$$
 Conversely, let $\cX$ be an $sncd$-model of $X$ such that $h(y)$
 lies in $\widehat{\cX}_\eta$. Then by applying Proposition \ref{prop-assu} to $Y$ and a compactification of $\cX$ to a proper $R$-model of $\overline{X}$,
 we can find a proper morphism of $sncd$-models $\cX'\to \cX$ and
 an open subscheme $\cY$ of $\cX'$
  such that $\cY$ is an $sncd$-model of $Y$ and $y\in \widehat{\cY}_\eta$.
 We have
 $$\weight_{\omega}(\rho_{\cX}(h(y))\leq \weight_{\omega}(\rho_{\cX'}(h(y))=\weight_{h^*\omega}(\rho_{\cY}(y))\leq \weight_{h^*\omega}(y)$$
and thus
$$\weight_{h^*\omega}(y)= \weight_{\omega}(h(y)).$$

\eqref{item:pluri} This is obvious.
\end{proof}

\begin{remark}
The weight function $$\weight_{\omega}\colon X^{\an}\to \R\cup
\{+\infty\}$$ is not continuous. Assume, for instance, that $X$ is
a curve and that $\cX$ is a regular $R$-model of $X$. Let $E$ be
an irreducible component of $\cX_k$ of multiplicity $N$, and let
$(y_n)_{n\geq 0}$ be any sequence of distinct closed points on $E$
that are not contained in any other irreducible component of
$\cX_k$. If we denote by $x_n$ the divisorial point of $X^{\an}$
associated to the exceptional divisor of the blow-up of $\cX$ at
$y_n$, then the sequence $(x_n)_{n\geq 0}$ converges to the
divisorial point $x\in X^{\an}$ associated to $(\cX,E)$. However,
 a direct computation shows that $$\weight_{\omega}(x_n)\geq
 \weight_{\omega}(x)+1/N$$ for every $n\geq 0$.
\end{remark}

\subsection{The Kontsevich-Soibelman skeleton}\label{subsec-KS}
\sss \label{sss-KSsk} Let $X$ be a connected smooth separated
$K$-scheme of dimension $n$ and let $\omega$ be a non-zero
rational $m$-pluricanonical form on $X$, for some $m>0$. We define
the weight of $X$ with respect to $\omega$ to be the value
$$\weight_{\omega}(X)=\inf\{\weight_{\omega}(x)\,|\,x\in
\mathrm{Div}(X)\}\in \R\cup\{-\infty\}.$$ We say that a divisorial
point $x$ of $X^{\an}$ is $\omega$-essential if
$$\weight_{\omega}(x)=\weight_{\omega}(X).$$
 The closure in $\mathrm{Bir}(X)$ (with the induced topology from $X^{\an}$) of the set of $\omega$-essential points is called the Kontsevich-Soibelman
skeleton of $(X,\omega)$ and denoted by $\Sk(X,\omega)$. We endow
it with the induced topology from $X^{\an}$.

\begin{prop}\label{prop-birat}
The skeleton $(X,\omega)$ is a birational invariant: if $h\colon Y\to X$
is a birational morphism of connected smooth $K$-schemes, then the
morphism $h^{\an}\colon Y^{\an}\to X^{\an}$ induces a homeomorphism
 between $\Sk(Y,h^*\omega)$ and $\Sk(X,\omega)$.
\end{prop}
\begin{proof}
It is enough to consider the case where $h$ is an open immersion.
This case is trivial, since $Y^{\an}$ and $X^{\an}$ have the same
divisorial and birational points and we can extend every model
$\cY$ for $Y$ to a model for $X$ by gluing $X$ to the generic
fiber $\cY_K$.
\end{proof}

\sss \label{sss-compar} The skeleton $\Sk(X,\omega)$ was
introduced by Kontsevich and Soibelman in \cite[\S6.6]{KS} in the
case where $R=\C\llbracket t\rrbracket$ and $X$ and $\omega$ are defined over the
ring $\C\{t\}$ of germs of analytic functions on $\C$ at $0$.
Instead of our definition of the weight, they used
 the following invariant:
 with the notations of \eqref{sss-wdiv}, they
 considered the value
 $$\frac{1}{N}\cdot \mathrm{ord}_{E}(\omega \wedge dt/t)$$ where $\omega \wedge
 dt$ is viewed as a meromorphic differential form of maximal
 degree on $\cX$. However, this invariant does not lead to the
 correct definition of the skeleton, and Theorem 3 in
 \cite[\S6.6]{KS} is incorrect as stated. Instead, one should use the
 invariant
$$\frac{1}{N}\cdot (\mathrm{ord}_{E}(\omega \wedge dt/t)+1)$$
 which coincides with our definition of weight up to shift by $-1$ because the wedge
 product with $dt$ defines an isomorphism between $\omega_{\cX/R}$ and
 $\Omega^{n+1}_{\cX/\C}$. We will now prove a generalization of Theorem 3 in
 \cite[\S6.6]{KS}. Instead of invoking the Weak Factorization
 Theorem as in \cite[\S6.6]{KS}, we only use the elementary properties of the weight function
 that we have proven in the preceding subsections.

\sss \label{sss-KS} Assume that $X$ is proper over $K$ and that
$\omega$ is regular at every point of $X$. We suppose furthermore
that $X$ has a proper $sncd$-model $\cX$. We
 will give an explicit description of $\Sk(X,\omega)$ in terms of $\cX$.
 We write $$\cX_k=\sum_{i\in I}N_i E_i$$ and we denote by
 $\mu_i-m$ the multiplicity of $E_i$ in
 $\mathrm{div}_{\cX}(\omega)$, for each $i$ in $I$.
 Let $J$ be a
non-empty subset of $I$ and let $\xi$ be a generic point of
$\cap_{j\in J} E_j$. We say that $\xi$ is $\omega$-essential if
$$\frac{\mu_j}{N_j}=\min\{\frac{\mu_i}{N_i}\,|\,i\in I\}$$ for
every $j$ in $J$ and $\xi$ is not contained in the closure of the
zero locus of $\omega$ on $X$.

\begin{theorem}\label{thm-KS} With the assumptions and notations of \eqref{sss-KS}, we have  $$\weight_{\omega}(X)=\min \{\weight_{\omega}(x)\,|\,x\in
\mathrm{Div}(X)\}=\min \{\weight_{\omega}(x)\,|\,x\in
X^{\an}\}=\min\{\frac{\mu_i}{N_i}\,|\,i\in I\}.$$  The
Kontsevich-Soibelman skeleton $\Sk(X,\omega)$ is equal to the set
of points $x\in X^{\an}$ where the weight function
$\weight_{\omega}$ reaches its minimal value
$\weight_{\omega}(X)$. It is the union of the open faces in the
Berkovich skeleton $\Sk(\cX)$ corresponding to the
 $\omega$-essential points of $\cX_k$. This is a non-empty compact
 subspace of $\Sk(\cX)$.

 In particular, this union of faces in the Berkovich skeleton $\Sk(\cX)$ only depends on $X$ and
 $\omega$, and not on the choice of the proper $sncd$-model $\cX$.
\end{theorem}
\begin{proof}
This follows immediately from Proposition \ref{prop-monweight} and
the fact that
$$\weight_{\cX,\omega}(x)\geq \sum_{i\in I}\mu_i v_x(E_i)$$ for
every $x$ in $\Sk(\cX)$, with equality if and only if
$\red_{\cX}(x)$ is not contained in the closure of the zero locus
of $\omega$ on $X$.
\end{proof}

\begin{example}
 Let $X$ be a smooth proper $K$-variety with trivial canonical
 sheaf, and assume that $X$ has a proper $sncd$-model $\cX$ such
 that $\omega_{\cX/R}$ is trivial. Then $\Sk(X)$ is the union of the
 closed faces of $\Sk(\cX)$ corresponding to the irreducible
 components of $\cX_k$ of maximal multiplicity. In particular, if
 $\cX_k$ is reduced, then $\Sk(X)=\Sk(\cX)$. Such models play an
 important role in the study of degenerations of complex
 $K3$-surfaces; see \cite{kulikov} and \cite{pers-pink}.
\end{example}

\sss  Note that $\Sk(X,\omega)$ may be empty if $\omega$ has poles
on $X$. For instance, it is easy to see that the weight of
$\Proj^1_K=\mathrm{Proj}\,K[x,y]$ with respect to the one-form
$d(x/y)$ is equal to $-\infty$.

\subsection{The essential skeleton}\label{ss-essential}
\sss Let $X$ be a connected smooth and proper $K$-variety. If
$\cX$ is an $sncd$-model of $X$, then the Berkovich skeleton
$\Sk(\cX)$ is of course highly dependent on the choice of $\cX$.
Nevertheless, we have seen in Theorem \ref{thm-KS} that for every
non-zero $m$-pluricanonical form $\omega$ on $X$, the
Kontsevich-Soibelman skeleton $\Sk(X,\omega)$ identifies certain
open faces of $\Sk(\cX)$ that must be contained in the Berkovich
skeleton of {\em every} $sncd$-model of $X$. This naturally leads
to the following definition.

\begin{definition}
 Let $X$ be a connected smooth and proper $K$-variety. We define
 the essential skeleton of $X$ by
 $$\Sk(X)=\bigcup_{\omega}\Sk(X,\omega)\quad \subset X^{\an}$$
 where $\omega$ runs through the set of non-zero regular
 pluricanonical forms on $X$.
\end{definition}

\begin{prop}
The essential skeleton $\Sk(X)$ is a birational invariant of $X$.
\end{prop}
\begin{proof}
This follows at once from the birational invariance of
$\Sk(X,\omega)$ and of the spaces of pluricanonical forms on $X$.
\end{proof}

\sss If $X$ has Kodaira dimension $-\infty$, then $\Sk(X)$ is
empty. If $X$ has Kodaira dimension $\geq 0$ and there exists a
proper $sncd$-model $\cX$ of $X$, then $\Sk(X)$ is a non-empty
 union of closed faces of $\Sk(\cX)$, by Theorem \ref{thm-KS} (it is a union of open faces and it is
 compact). In particular, $\Sk(X)$ carries a canonical piecewise
 affine structure, by Corollary \ref{cor-pieceaff}. If the
 canonical sheaf $\omega_{X/K}$ of $X$ is trivial, then
 $\Sk(X)=\Sk(X,\omega)$ where $\omega$ is any non-zero
 differential form of maximal degree on $X$.

\section{The skeleton is connected}\label{sec-connected} In this section we show that
the Kontsevich-Soibelman skeleton of a smooth and proper
$K$-variety of geometric genus one is connected, assuming that $k$
has characteristic zero. The proof is based on a generalization of
Koll\'ar's torsion-free theorem to $R$-schemes. We will first
deduce this generalization by means of an approximation argument.
\subsection{Algebraic approximation of
$R$-schemes}\label{subsec-Greenberg} \sss If $X$ is a proper
scheme of pure dimension $d$ over a field $F$, then a line bundle
$\mathcal{L}$ on $X$ is called big if there exists a constant
$C>0$ such that $$\mathrm{dim}_FH^0(X,\mathcal{L}^n)\geq Cn^d$$
for all sufficiently divisible positive integers $n$.

\begin{prop}\label{prop-approx}
Assume that $R$ has equal characteristic. Let $\cX$ be a flat
$R$-scheme and let $\mathcal{L}$ be a line bundle on $\cX$. Then
for every integer $n>0$, we can find the following data:
\begin{enumerate}
\item \label{item:curve} a smooth connected algebraic $k$-curve
$C$, a $k$-rational point $O$ on $C$ and a ring isomorphism
$$R/\frak{m}^n\to \mathcal{O}_{C,O}/\frak{m}^n_{C,O}$$ where
$\frak{m}_{C,O}$ denotes the maximal ideal of $\mathcal{O}_{C,O}$;
 \item \label{item:models} a  flat $C$-scheme $\cX'$, a line
bundle $\mathcal{L}'$ on $\cX'$, an isomorphism of
$R/\frak{m}^n$-schemes
$$f\colon \cX'\times_C \Spec (\mathcal{O}_{C,O}/\frak{m}^n_{C,O})\to
\cX\times_R (R/\frak{m}^n)$$ and an isomorphism of line bundles
$$\mathcal{L}'/\frak{m}^n_{C,O}\to
f^*(\mathcal{L}/\frak{m}^n).$$ Here we wrote
$\mathcal{L}'/\frak{m}^n_{C,O}$ for the pullback of $\mathcal{L}'$
to $\cX'\times_C \Spec (\mathcal{O}_{C,O}/\frak{m}^n_{C,O})$, and
$\mathcal{L}/\frak{m}^n$ for the pullback of $\mathcal{L}$ to
$\cX\times_R (R/\frak{m}^n)$.
\end{enumerate}
These data satisfy the following properties.
\begin{enumerate}\setcounter{enumi}{2}
\item \label{item:proper} If $\cX$ is proper over $R$, then we can
arrange that $\cX'$ is proper over  $C$.

\item \label{item:ample} If $\cX$ is projective over $R$ and
$\mathcal{L}$ is ample, then $\mathcal{L}'$ is ample over some
neighbourhood of $O$ in $C$.

\item \label{item:genample} If
 the restriction of $\mathcal{L}$ to $\cX_K$ is ample, then we can arrange that $\mathcal{L}'$ is ample over the generic point of
 $C$.
 If $\cX_K$ is proper over $K$ and the restriction of $\mathcal{L}$ to $\cX_K$ is big,
 then we can arrange that the generic fiber of $\cX'\to C$ is proper and that $\mathcal{L}'$ is big over the generic point of
 $C$.

\item \label{item:regular} If $\cX$ is regular and $n\geq 2$, then
$\cX'$ is regular at every point of the fiber over $O$.
 If, in addition, $E$ is a divisor on $\cX$ whose support is contained in $\cX_k$ and
   has strict normal crossings, then $E$ viewed as a divisor on $\cX'$ via the isomorphism of $k$-schemes $\cX_k\cong \cX'\times_C O$
    also has strict
normal crossings. If $\cX$ is regular and proper over $R$, then we
can arrange that $\cX'$ is regular and proper over $C$.
\end{enumerate}
\end{prop}
\begin{proof}
 We fix a $k$-algebra structure on $R$ by choosing a section of the projection morphism $R\to k$.
   By a standard spreading out argument \cite[\S8]{ega4.3}, we can find the following data:
\begin{itemize}
\item  a sub $k$-algebra $A$ of $R$ such that $A$ is integrally
closed
 and of finite type over $k$;

\item an $A$-scheme $\cX_A$ of finite type, a line bundle
$\mathcal{L}_A$ on $\cX_A$, an isomorphism of $R$-schemes
$$g\colon \cX\to \cX_A\times_A R$$ and an isomorphism of line bundles
$$g^*(\mathcal{L}_A\otimes_A R)\to \mathcal{L}$$ where we wrote
$\mathcal{L}_A\otimes_A R$ for the pullback of $\mathcal{L}_A$ to
$\mathcal{L}_A\times_A R$.
\end{itemize}
We denote by $a$ the image of the closed point of $\Spec R$ under
the morphism $$\Spec R\to \Spec A.$$ It follows
 from \cite[11.6.1]{ega4.3} that $\cX_A$ is flat over some open
 neighbourhood of $a$ in $\Spec A$. Replacing $A$ by a suitable
 localization, we may suppose that $\cX_A$ is flat over $A$.

 We denote by $R'$ the henselization of the local ring of $\A^1_k$ at the origin, and by $\frak{m}'$ its maximal ideal.
 The ring $R'$ is an excellent henselian discrete valuation ring whose
 completion is isomorphic to $R$; we fix such an isomorphism. By Greenberg's Approximation
 Theorem \cite[Thm.1]{Gr}, we can find a ring morphism $A\to R'$ such that
 the composition with the projection $R'\to R'/(\frak{m}')^n\cong R/\frak{m}^n$
 coincides with the morphism $A\to R/\frak{m}^n$ induced by the inclusion $A\to R$.
 We set
 $$\cX_{R'}=\cX_{A}\times_A R'$$ and we denote by
 $\mathcal{L}_{R'}$ the pullback of $\mathcal{L}_A$ to $\cX_{R'}$.

 The henselization $R'$ is a direct limit of
 local rings $\mathcal{O}_{C,O}$ where $C$ is an \'etale
 $\A^1_k$-scheme and $O$ is a $k$-rational point of $C$ lying over
 the origin of $\A^1_k$. Since $A$ is of finite type over $k$, we can find such a curve $C$ such that the morphism $\Spec R'\to \Spec A$ factors through a $k$-morphism $C\to \Spec A$.
 If we set $\cX'=\cX_A\times_A C$ and we denote by
 $\mathcal{L}'$ the pullback of $\mathcal{L}_A$ to $\cX'$, then
 the triple $(C,\cX',\mathcal{L}')$
 satisfies all the properties in \eqref{item:curve} and \eqref{item:models}.

Now we prove the list of properties in the second part of the
statement.

\eqref{item:proper} If $\cX$ is proper over $R$, then we can
 choose $A$ and $\cX_A$ in such a way that $\cX_A$ is proper over $A$ by \cite[8.10.5]{ega4.3};
then $\cX'$ is proper over $C$.

\eqref{item:ample} If $\cX$ is projective over $R$ and
$\mathcal{L}$ is ample, then $\mathcal{L}/\frak{m}$ is ample on
$\cX_k$,
 and it follows from \cite[4.7.1]{ega3.1} that $\mathcal{L}'$ is ample over some open
 neighbourhood of $O$ in $C$.

\eqref{item:genample} Assume that the restriction of $\mathcal{L}$
to $\cX_K$ is ample. Then by flat descent of ampleness
\cite[2.7.2]{ega4.2}, $\mathcal{L}_A$ is ample over the generic
point of $\Spec A$. Thus there exists a dense open subscheme $U$
of $\Spec A$ over which $\mathcal{L}_A$ is ample, by
\cite[8.10.5.2]{ega4.3}. We denote its complement by $Z$. To prove
our statement, we can always replace $n$ by a larger integer $n'$,
since the result will then also be valid for $n$. For $n'$
sufficiently large, the morphism $\Spec R/\frak{m}^{n'}\to \Spec
A$ does not factor though $Z$, since $\Spec R\to \Spec A$ is
dominant. Thus if $n'$ is large enough given our choice of $A$,
$\cX_A$ and $\mathcal{L}_A$, then the
 morphism $C\to \Spec A$ maps the generic point of $C$ to a point
 of $U$, which implies that $\mathcal{L}'$ is ample over the
 generic point of $C$.
  The statement for big line
 bundles can be proved in a similar way: the generic fiber of $\cX_A\to \Spec A$ is proper by flat descent of properness, and the line bundle $\mathcal{L}_A$ is big over the generic point of $\Spec
 A$ by flat base change for coherent cohomology. Then
 $\cX_A$ is proper over a dense open subset $U$ of $\Spec A$, and
 $\mathcal{L}_A$ is big over
 every point in $U$ by upper-semicontinuity of the function $$U\to \N,\,x\mapsto h^0(\cX_A\times_A x,\mathcal{L}^d_A\otimes_A \kappa(x))$$ for all $d>0$ \cite[7.7.5]{ega3.2}.

\eqref{item:regular} If $n\geq 2$, then the reduction of $\cX$
modulo
 $\frak{m}^2$ is isomorphic to the reduction of $\cX'$ modulo
 $\frak{m}_{C,O}^2$. Thus the Zariski tangent space of $\cX$ at
 any point of $\cX_k$ has the same dimension as the Zariski
 tangent space of $\cX'$ at the corresponding point of $\cX'\times_{C}O\cong
 \cX_k$. It follows that $\cX'$ is regular at every point over
 $O$ if $\cX$ is regular. The statement about $E$ is obvious. If $\cX$ is regular and proper over $R$, then we can take $\cX'$ to be proper over $C$ and regular at every point in the fiber over $O$. Then the
 image in $C$ of the
 singular locus of $\cX'$ is a closed subset of $C$ that does not
 contain $O$. Thus by shrinking $C$, we can arrange that $\cX'$ is
 regular.
\end{proof}

\subsection{Relative vanishing theorems over a ring of formal power series}
\sss Assume that $k$ has characteristic zero. We will extend some
relative vanishing theorems for morphisms of complex varieties to
the category of schemes of finite type over $R$. The
 standard proofs of these vanishing results use transcendental
methods and cannot be adapted to $R$-schemes in a direct way. So
we will use an approximation argument to reduce to the case where
our $R$-scheme is defined over an algebraic curve.
 In the proof of the connectedness theorem, we will only use
the torsion-free theorem (Theorem \ref{thm-torfree}), but we
decided to include the Kawamata-Viehweg vanishing theorem (Theorem
\ref{thm-vanish}) as well, because it is of independent interest
and the method of proof is similar. Theorem \ref{thm-vanish} and
Corollary \ref{cor-kodaira}
 were
also proven in Appendix B of \cite{BFJ}, but under more
restrictive conditions (there it was assumed that $\cX_k$ is a
strict normal crossings divisor) and with more complicated
arguments.


 \sss For notational convenience, we will denote $R/\frak{m}^n$ by
 $R_n$, and $\cX\times_R (R/\frak{m}^n)$ by $\cX_n$, for every
 integer $n>0$ and every $R$-scheme $\cX$. In particular, $R_0=k$ and $\cX_0=\cX_k$. If $\cX$ is normal and $E$ is a Cartier divisor on $\cX$, then we denote by
 $\mathcal{O}_{\cX_n}(E)$ the pullback of the line bundle
 $\mathcal{O}_{\cX}(E)$ to $\cX_n$.

\begin{theorem}[Kawamata-Viehweg vanishing]\label{thm-vanish}
Assume that $k$ has characteristic zero. Let $\cX$ be a regular
 projective flat $R$-scheme, and denote by $\omega_{\cX/R}$ its relative
canonical sheaf. Let $\Delta$ be a $\Q$-divisor on $\cX$,
supported on the special fiber $\cX_k$, such that $\Delta$ has
strict normal crossings and all multiplicities of its prime
divisors lie in $[0,1[$. Let $E$ be a divisor on $\cX$
 such that the $\Q$-divisor $E-\Delta$ is
 relatively nef and such that the restriction of $E$ to $\cX_K$ is ample. 
 Then
$$H^i(\cX,\omega_{\cX/R}\otimes \mathcal{O}_{\cX}(E))=0$$ for all
$i>0$.
\end{theorem}
\begin{proof}
 Set $\mathcal{L}=\mathcal{O}_{\cX}(E)$. We
set $n=1$ and we choose $C$, $\cX'$ and $\mathcal{L}'$ as in
Proposition \ref{prop-approx}, such that $\cX'$ is regular and
proper over $C$ and $\mathcal{L}'$ is ample on the generic fiber
of $\cX'\to C$. We choose a divisor $E'$ on $\cX'$ in the linear
equivalence class defined by the line bundle $\mathcal{L}'$.
Shrinking $C$, we may assume that $C$ is affine and $\mathcal{L}'$
is relatively ample over $C\smallsetminus \{O\}$, by
\cite[8.10.5.2]{ega4.3}. Then $E'-\Delta$ is relatively nef, since
the restriction of $\mathcal{L}'$ to $\cX'\times_{C}O\cong \cX_k$
is isomorphic to the restriction of $\mathcal{L}$ to $\cX_k$.

 By the relative version of the Kawamata-Viehweg
vanishing theorem \cite[2.17.3]{kollar}, we have
$$H^i(\cX',\omega_{\cX'/C}\otimes \mathcal{L}')=0$$ for
all $i>0$.
 Let
$t$ be a uniformizer in $\mathcal{O}_{C,O}$. Shrinking $C$, we may
assume that $t$ is regular at every point of $C$. Looking at the
long exact cohomology sequence associated to the short exact
sequence of $\mathcal{O}_{\cX'}$-modules
\begin{equation}\label{eq-seq}
\minCDarrowwidth15pt\begin{CD}0@>>> \omega_{\cX'/C}\otimes
\mathcal{L}'@>t\cdot
>> \omega_{\cX'/C}\otimes \mathcal{L}'@>>>
 \omega_{\cX_k/k}\otimes
\mathcal{O}_{\cX_k}(E)@>>> 0,\end{CD}\end{equation}
 we find that $$H^i(\cX_k,\omega_{\cX_k/k}\otimes
\mathcal{O}_{\cX_k}(E))=0$$ for all $i>0$. Now it follows from the
semicontinuity theorem for coherent cohomology
\cite[7.7.5]{ega3.2}, combined with the implication $e)\Rightarrow
d)$ in \cite[7.8.4]{ega3.2},
 that
$$H^i(\cX,\omega_{\cX/R}\otimes
\mathcal{O}_{\cX}(E))=0$$ for all $i>0$.
\end{proof}

\begin{cor}[Kodaira vanishing]\label{cor-kodaira} Assume that $k$ has characteristic
zero. Let $\cX$ be a regular flat projective $R$-scheme, and
denote by $\omega_{\cX/R}$ its relative canonical sheaf. Then for
every ample line bundle $\mathcal{L}$ on $\cX$ and every integer
$i>0$, we have
$$H^i(\cX,\omega_{\cX/R}\otimes\mathcal{L})=0.$$
\end{cor}
\begin{proof}
This is a special case of Kawamata-Viehweg vanishing, with
$\Delta=0$ and $E$ the Cartier divisor associated to
$\mathcal{L}$.
\end{proof}

\begin{remark}
One can generalize Theorem \ref{thm-vanish} by requiring  that $E$
is big on $\cX_K$ instead of ample, as in the Kawamata-Viehweg
vanishing theorem formulated in \cite[2.17.3]{kollar}. By
Proposition \ref{prop-approx},  we can find $\cX'$, $C$ and
$\mathcal{L}'$ as in the proof of Theorem \ref{thm-vanish} such
that $\mathcal{L}'$ is big on the generic fiber of $\cX'\to C'$.
But
 we do not know if one can arrange moreover that $\mathcal{L}'$ is
 relatively nef. The problem is that nefness behaves poorly in
 families: even if $\mathcal{L}'$ is nef over $O$ and over the
 generic point of $C$, this does not guarantee that $\mathcal{L}'$
 is nef over some open neighbourhood of $O$ in $C$. However, one
 can circumvent this problem
by means of the following variant of the Kawamata-Viehweg
vanishing theorem. Let $k$ be a field of characteristic zero and
let $f\colon X\to S$ be a surjective projective morphism of
schemes of finite type over $k$, with $X$ smooth and connected.
 Let $s$ be a closed point on  $S$. Let $\Delta$ be a $\Q$-divisor on $X$
 such that $\Delta$ has strict normal
crossings and all multiplicities of its prime divisors lie in
$[0,1[$. Let $E$ be a divisor on $X$
 such that the restriction of the $\Q$-divisor $E-\Delta$ to $X_s=f^{-1}(s)$
 is nef and such that the restriction of $E-\Delta$ to the generic fiber of $f$ is
big. Then we have $R^if_*\mathcal{O}_X(K_X+E)_s=0$ for every
$i\geq 1$.
 This variant of the Kawamata-Viehweg
vanishing theorem can be proven with the same arguments as in
\cite[Theorem~1-2-3]{KMM}.
\end{remark}

\begin{theorem}[Torsion-free theorem]\label{thm-torfree}
Assume that $k$ has characteristic zero.  Let $\cX$ be a regular
proper flat $R$-scheme, and denote by $\omega_{\cX/R}$ its
relative canonical sheaf. Let $\Delta$ be a $\Q$-divisor on $\cX$,
supported on the special fiber $\cX_k$, such that $\Delta$ has
strict normal crossings and all multiplicities of its prime
divisors lie in $[0,1[$. Let $E$ be a divisor on $\cX$ such that
the $\Q$-divisor $E-\Delta$ is a rational multiple of $\cX_k$.
 Then
$$H^i(\cX,\omega_{\cX/R}\otimes \mathcal{O}_{\cX}(E))$$ is a free $R$-module for all
$i\geq 0$.
\end{theorem}
\begin{proof}
\if false
 Let $(M_n)_{n>0}$ be a projective system of $R$-modules.
We denote by $M$ its projective limit. Moreover, we denote for
every $n> 0$ by $M'_n$ the submodule of $M_n$ consisting of
elements killed by $t$.  By left exactness of the projective
limit, we know that
$$M'=\lim_{\stackrel{\longleftarrow}{n}}M_n'$$ is the submodule of
$M$ consisting of the elements killed by $t$.  Assume that the
projective system $(M'_n)_{n>0}$ is essentially zero, which means
that for each $n>0$, there exists an integer $m>n$ such that the
transition morphism $M_m\to M_n$ is the zero map. Then $M'=0$, so
that $M$ has no $t$-torsion. \fi
 We fix an integer $i\geq0$ and
set
$$M_n=H^i(\cX_n,\omega_{\cX_n/R_n}\otimes \mathcal{O}_{\cX_n}(E))$$ for all
$n>0$. These $R$-modules form a projective system, whose limit is
precisely the cohomology module
$$M=H^i(\cX,\omega_{\cX/R}\otimes \mathcal{O}_{\cX}(E))$$ by
Grothendieck's comparison theorem \cite[4.1.5]{ega3.1}.
 Fix an integer $n>1$,
and choose $C$ and $\cX'$ as in Proposition \ref{prop-approx},
with $\cX'$ regular and
 proper over $C$. Then we can view $E$ as a divisor on $\cX'$, supported on the fiber over $O$, which is isomorphic to $\cX_k$. We denote by $f\colon \cX'\to C$ the structural
morphism. Koll\'ar's torsion-free theorem
\cite[2.17.4]{kollar} implies that
$$R^jf_*(\omega_{\cX'/C}\otimes \mathcal{O}_{\cX'}(E))$$ is
 locally free for all $j\geq 0$.
 Thus the short exact sequence \eqref{eq-seq} induces an
 exact sequence
 $$R^if_*(\omega_{\cX'/C}\otimes \mathcal{O}_{\cX'}(E))\to R^if_*(\omega_{\cX'/C}\otimes \mathcal{O}_{\cX'}(E))\to M_n\to 0$$
(the coboundary map must be zero because $M_n$ is a torsion
module).
 Hence, we get a canonical isomorphism
$$M_n\cong R^if_*(\omega_{\cX'/C}\otimes \mathcal{O}_{\cX'}(E))\otimes_{\mathcal{O}_C}(\mathcal{O}_{C,O}/\frak{m}_{C,O}^n).$$
 and
  $M_n$ is a free $R_n$-module for all $n>0$. This implies that $M$ has no
  $t$-torsion, because every element in $M_{n+1}$ killed by $t$ is
  mapped to $0$ in $M_n$.
\end{proof}

\begin{remark} With a little extra work, Theorem \ref{thm-torfree}
can be generalized: one can drop the assumption that $\Delta$ and
$E$ are supported on the special fiber $\cX_k$, and require
 only that $E-\Delta$ is $\Q$-linearly equivalent to a rational multiple
of $\cX_k$ (instead of equal). Indeed, by choosing $A$
sufficiently large in Proposition \ref{prop-approx} we can assume
that $E$, $\Delta$ and the $\Q$-linear equivalence are defined
over $A$, and by shrinking $C$ we can assume that the pullback of
$\Delta$ to $\cX'$ has strict normal crossings, so that we still
can apply Koll\'ar's torsion-free theorem.
\end{remark}

\subsection{The skeleton is connected}
\sss We will now show that the skeleton of a  smooth, proper,
geometrically connected $K$-variety of geometric genus one is
always connected. We will explain in \eqref{sss-shoko} how this
result can be viewed as a global version of the connectedness
theorem of Koll\'ar and Shokurov. The proof is somewhat different:
we cannot use Kawamata-Viehweg vanishing because the divisor $D$
that appears in the proof is not big on the generic fiber $\cX_K$.
Instead, we will apply the torsion-free theorem.

\sss We will use the following terminology for divisors on
$R$-schemes. Let $\cX$ be a flat
 $R$-scheme of finite type, and let $D$ be a Weil divisor on $\cX$.
 Then we can write $$D=D^{+}-D^{-}$$ such that $D^{+}$ and $D^{-}$
 are effective and have no common prime divisors. We call $D^{+}$
 and $D^{-}$ the positive, resp.~negative part of $D$.

\begin{theorem}\label{thm-connected}
Assume that $k$ has characteristic zero. Let $X$ be a proper
smooth geometrically connected $K$-variety
 of geometric genus one, and let $\omega$ be a non-zero differential form of
maximal degree on $X$. Then for every $sncd$-model $\cX$ of $X$,
the union of the $\omega$-essential components of $\cX_k$ is
connected.
\end{theorem}
\begin{proof}
  We write
$$\cX_k=\sum_{i\in I}N_i E_i$$ and we denote by $\mu_i-1$ the
multiplicity of $E_i$ in $\mathrm{div}_{\cX}(\omega)$, for each
$i\in I$. We set
$$\alpha=\min\{\mu_i/N_i\,|\,i\in I\}.$$
 We denote by $\Delta$ the fractional part of the $\Q$-divisor
 $\alpha\cX_k$. Then, by definition
 of $\alpha$, none of the prime components of $\Delta$
 is $\omega$-essential. Moreover, the negative part $D^{-}$ of the divisor
$$D=\mathrm{div}_{\cX}(\omega)-\alpha\cX_k+\Delta$$ is reduced, and its support is precisely the union of
 the $\omega$-essential components of $\cX_k$. Thus we must show that $D^{-}$ is connected.

 From the short exact
sequence of $\mathcal{O}_{\cX}$-modules
$$0\to \mathcal{O}_{\cX}(D)\to \mathcal{O}_{\cX}(D^{+})\to
\mathcal{O}_{D^{-}}(D^{+})\to 0$$ we obtain an exact sequence of
cohomology $R$-modules
$$\minCDarrowwidth15pt\begin{CD}H^0(\cX,\mathcal{O}_{\cX}(D^{+})) @>f>>
H^0(D^{-},\mathcal{O}_{D^{-}}(D^{+}))@>>>
H^1(\cX,\mathcal{O}_{\cX}(D)).\end{CD}$$ By Theorem
\ref{thm-torfree}, we know that the $R$-module
$$H^1(\cX,\mathcal{O}_{\cX}(D))$$ is free. On the other hand, the cokernel of
$$f\colon H^0(\cX,\mathcal{O}_{\cX}(D^{+})) \to
H^0(D^{-},\mathcal{O}_{D^{-}}(D^{+}))$$ is a torsion $R$-module.
As it injects into a free $R$-module, it must be zero, and $f$  is
surjective.

The restriction of $D^+$ to $X=\cX_K$ is a canonical divisor.
Since the geometric genus of $X$ is one,
 we have
$H^0(X,\mathcal{O}_{\cX}(D^{+})|_{X})=K$ and thus
$$H^0(\cX,\mathcal{O}_{\cX}(D^{+}))\cong R.$$
Now the surjectivity of $f$ implies that the $k$-vector space
$$H^0(D^{-},\mathcal{O}_{D^{-}}(D^{+}))$$ has rank one.
 But all locally constant
 $k$-valued functions on $D^{-}$ belong to this space, so that every locally constant function on $D^{-}$ is constant and $D^{-}$ is connected.
\end{proof}
\begin{cor}[Connectedness theorem]\label{cor-connected}
Assume that $k$ has characteristic zero. Let $X$ be a proper
smooth geometrically connected $K$-variety of geometric genus one,
and let $\omega$ be a non-zero differential form of maximal degree
on $X$.
 Then the skeleton $\mathrm{Sk}(X,\omega)$
of $X$ is connected.  In particular, if $\omega_{X/K}$ is trivial,
then the essential skeleton $\Sk(X)$ of $X$ is connected.
\end{cor}
\begin{proof}
If $\omega_{X/K}$ is trivial, then this is a direct consequence of
Theorems \ref{thm-KS} and \ref{thm-connected}. In the general
case, we need to be more careful: if $\omega$ is a non-zero
differential form of maximal degree on $X$, $\cX$ is an
$sncd$-model of $X$ and $Z$ is the schematic closure in $\cX$ of
the zero locus of $\omega$ in $X$, then an intersection of two
$\omega$-essential irreducible components of $\cX_k$ does not
contribute to the skeleton if it is contained in
 $Z$. However, the
 Kontsevich-Soibelman skeleton is a birational invariant by
Proposition \ref{prop-birat}, and we can always find a proper
birational $R$-morphism $h\colon \cX'\to \cX$ such that $\cX'$ is
regular, $\cX'_K$ is smooth and proper over $K$, $h$ is an
isomorphism over $\cX\smallsetminus Z$ and $ \cX'_k+h^{*}Z$
 is a divisor with strict normal crossings on
 $\cX'$. Note that $\cX'_K$ has geometric genus one, because the
 geometric genus is a birational invariant. Moreover, the closure
 $Z'$ in $\cX'$ of the locus of zeroes of $\omega$ on $\cX'_K$ is
 contained in $h^*Z$.
 Thus the  normal crossings property implies that $Z'$ cannot contain a connected component of the intersection of two
 irreducible components of $\cX'_k$. Hence, it follows from Theorems \ref{thm-KS} and
 \ref{thm-connected} that $\Sk(X,\omega)=\Sk(\cX'_K,h^*_K\omega)$ is connected.
\end{proof}

\section{Relation with the birational geometry of
varieties over a field of characteristic zero}\label{sec-birat}
 Our definition of the weight function
has a natural counterpart in birational geometry, that we will
 explain in this section. It is closely related to the thinness function constructed in
 \cite{BFJ0} and the log
discrepancy function in \cite{jonsson-mustata}.
\subsection{The weight function of a coherent ideal
sheaf}\label{ss-birat}
 \sss \label{sss-birat}
Let $F$ be a field of characteristic zero. Let $X$ be a connected
smooth $F$-variety and let $\mathcal{I}$ be a nonzero coherent ideal sheaf
on $X$. Let $v$ be a divisorial valuation on the function field of
$X$ such that $v$ has a center on $X$ and this center lies in the
zero locus $Z(\mathcal{I})$ of $\mathcal{I}$. Then one can find a
log resolution $h\colon Y\to X$ of $\mathcal{I}$ such that the closure
of the center of $v$ on $Y$ is a divisor $E$ on $Y$. Recall that
this means that $h$ is a proper birational morphism such that
$Z(\mathcal{I}\mathcal{O}_Y)+K_{Y/X}$ is a divisor with strict
normal crossings on $Y$. One can moreover arrange that $h$ is an
isomorphism over $X\smallsetminus Z(\mathcal{I})$, and we will always
assume that a log resolution satisfies this property. The
valuation $v$ is equal to $r\cdot \mathrm{ord}_E$, where
$\mathrm{ord}_E$ is the valuation associated to the divisor $E$
and $r$ is a positive real number. If we denote by $N$ the
multiplicity of $E$ in $Z(\mathcal{I}\mathcal{O}_Y)$ and by
$\mu-1$ the multiplicity of $E$ in
 $K_{Y/X}$, then the quotient $\mu/N$ is an interesting geometric
 invariant, which we call the weight of $v$ with respect to the
 variety $X$ and which we denote by $\weight_{\mathcal{I}}(v)$. Note that it does not depend on $r$. The set of the weights of all divisorial valuations
 $v$ with center in $Z(\mathcal{I})$ has a minimum, called the log canonical threshold of the pair
 $(X,\mathcal{I})$ and denoted by $\lct(X,\mathcal{I})$. This is a fundamental invariant
 in birational geometry; see \cite[\S8]{kollar}. It is well-known
 that the log canonical threshold can be computed on a single
 log resolution:
  if $Y\to X$ is any log resolution of $\mathcal{I}$, then writing $Z(\mathcal{I}\mathcal{O}_Y)=\sum_{i\in I}N_iE_i$ and $K_{Y/X}=\sum_{i\in I}(\mu_i-1)E_i$, one has
 $$\lct(X,\mathcal{I})=\min\{\frac{\mu_i}{N_i}\,|\,i\in I\}.$$
 This can be viewed as a
   partial analog of Theorem
 \ref{thm-KS}.

 \sss \label{sss-qmon} The weight function can be extended to
 quasi-monomial valuations on $X$ with center in $Z(\mathcal{I})$
 (see \cite[3.1]{jonsson-mustata} for the definition of a
 quasi-monomial valuation). If $v$ is such a quasi-monomial
 valuation, then we set
 $$\weight_{\mathcal{I}}(v)=\frac{v(K_{Y/X}+Z(\mathcal{I}\mathcal{O}_Y)_{\red})}{v(Z(\mathcal{I}\mathcal{O}_Y))}$$
 where $h\colon Y\to X$ is a log resolution of $\mathcal{I}$ such that
 $(Y,Z(\mathcal{I}\mathcal{O}_Y)_{\red})$ is adapted to $v$ in the sense of
 \cite[3.5]{jonsson-mustata} and where, for every effective divisor $D$ on $Y$, we
 write
 $$v(D)=\min\{v(f)\,|\,f\in \mathcal{O}(-D)_{\xi}\}$$ with $\xi$
 the center of $v$ on $Y$.  One can show that the definition of $\weight_{\mathcal{I}}(v)$ does not
 depend on the choice of $h$; see
 Section \ref{ss-comparJM}.



 \sss \label{sss-weightext2} For every $X$-scheme of finite type $Y$, we denote by $\widehat{Y}$ the formal
 $\mathcal{I}\mathcal{O}_Y$-adic completion of $Y$. In particular,
  $\widehat{X}$ is the formal $\mathcal{I}$-adic completion of
 $X$.
  We consider the generic fiber $\widehat{X}_\eta$ in the sense of
  \cite[1.7]{thuillier}.
 This is an analytic
space over the field $F$ endowed with its trivial absolute value.
It is obtained by removing from the usual generic fiber of
$\widehat{X}$ all the points that lie on the analytification of
the closed subscheme $Z(\mathcal{I})$ of $\widehat{X}$. It carries
a natural reduction map $$\red_{\widehat{X}}\colon \widehat{X}_\eta\to
Z(\mathcal{I}).$$ The set of quasi-monomial valuations $v$ on $X$
with center in $Z(\mathcal{I})$ can be embedded in
$\widehat{X}_\eta$ by sending $v$ to the absolute value $f\mapsto
\exp(-v(f))$ on the function field $F(X)$ of $X$. We will call the
points in the image quasi-monomial points.

\sss We can extend $\weight_{\mathcal{I}}$ to a function
$$\weight_{\mathcal{I}}\colon \widehat{X}_\eta\to \R\cup\{+\infty\}$$
in a similar way as in \eqref{sss-weightext}. Each log resolution
$Y\to X$ of $\mathcal{I}$ gives rise to a skeleton
$\Sk(\widehat{Y})\subset \widehat{X}_\eta$, consisting of the
quasi-monomial points $x$ in $\widehat{X}_\eta$ such that $Y$ is
adapted to the corresponding valuation $v_x$.
 The skeleton $\Sk(\widehat{Y})$ is homeomorphic to the product of $\R_{>0}$ with the
simplicial space associated to the strict normal crossings divisor
$Z(\mathcal{I}\mathcal{O}_Y)$, and there exists a natural
contraction $\rho_{Y}\colon \widehat{X}_\eta\to \Sk(\widehat{Y})$ that
can be extended to a strong deformation retraction of
$\widehat{X}_\eta$ onto $\Sk(\widehat{Y})$ \cite[3.27]{thuillier}.
 One can prove that
 \begin{equation}\label{eq-qmonineq}
 \weight_{\mathcal{I}}(x)\geq \weight_{\mathcal{I}}(\rho_Y(x))
 \end{equation}
 for every quasi-monomial point $x$ on $\widehat{X}_\eta$, with
 equality if and only if $x$ lies in $\Sk(\widehat{Y})$; see
 Section \ref{ss-comparJM}.
 We define the weight function on $\widehat{X}_\eta$ as
  $$\weight_{\mathcal{I}}\colon \widehat{X}_\eta\to \R\cup \{+\infty\},\,x\mapsto \weight_{\mathcal{I}}(x)=\sup_{h\colon Y\to X}\weight_{\mathcal{I}}(\rho_Y(x))$$
   where the
  supremum is taken over all log resolutions $h\colon Y\to X$ of
  $\mathcal{I}$. This weight function satisfies properties that
  are quite similar to the ones in Proposition
  \ref{prop-weightext}.

\subsection{Comparison with the log discrepancy
function}\label{ss-comparJM}
 \sss In \cite{jonsson-mustata}, Mattias Jonsson and the
 first-named author studied the properties of the so-called log discrepancy function $A_X$ on a
 certain space of valuations. This function had previously been
 introduced as the {\em thinness function} in a slightly different setting
 \cite{BFJ}.
 We will now explain the relation
 with the weight function defined above. As in
 \cite{jonsson-mustata}, we denote by $\mathrm{Val}_X$ the space of real valuations on the function field $F(X)$ of $X$ with
 a center on $X$. We can view $\mathrm{Val}_X$ as a subspace of the analytification
 $X^{\an}$ of $X$ with respect to the trivial absolute value on
 $F$ by associating the absolute value $f\mapsto \exp(-v(f))$ on
 $F(X)$ to each valuation $v$ in $\mathrm{Val}_X$.

 \sss For every log resolution $h\colon Y\to X$ of
 $\mathcal{I}$, the skeleton $\Sk(\widehat{Y})$ is contained in
 $\mathrm{Val}_X$.
 In
 \cite{jonsson-mustata}, the union of $\Sk(\widehat{Y})$ and the trivial valuation on $F(X)$ was denoted by
 $\mathrm{QM}(Y,Z(\mathcal{I}\mathcal{O}_Y))$. With our notation,
 the restriction of the log discrepancy function to
 $\Sk(\widehat{Y})$ is given by
 $$A_X\colon \Sk(\widehat{Y})\to \R,\,x\mapsto v_x(K_{Y/X}+Z(\mathcal{I}\mathcal{O}_Y)_{\red})$$
 where $v_x$ is the quasi-monomial valuation corresponding to $x$.
 It is proven in \cite[5.1]{jonsson-mustata} that $A_X(x)$ only
 depends on the point $x$, and not on the choice of the log
 resolution $Y\to X$ such that $\Sk(\widehat{Y})$ contains $x$, so
 that $A_X$ is well-defined on the set of quasi-monomial points in
 $\widehat{X}_\eta$.
 This implies the analogous claim for $\weight_{\mathcal{I}}$ in
 \eqref{sss-qmon}, because
 $v_x(Z(\mathcal{I}\mathcal{O}_Y))$ does not depend on the chosen
 log resolution. Likewise, the inequality \eqref{eq-qmonineq}
 follows from \cite[5.3]{jonsson-mustata}.

\sss In \cite[\S5.2]{jonsson-mustata}, the log discrepancy
function is extended to $\mathrm{Val}_X$ by means of a supremum
construction similar to the one we used in \eqref{sss-weightext2}.
It follows immediately from the definitions that we have
$$\weight_{\mathcal{I}}(x)=\frac{A_X(x)}{N_x(\mathcal{I})}$$
for every $x$ in $\mathrm{Val}_X\cap \widehat{X}_\eta$, where
$N_x(\mathcal{I})$ denotes the value
$v_x(Z(\mathcal{I}\mathcal{O}_Y))$
 for any log resolution $Y\to X$ of $\mathcal{I}$.

\subsection{Comparison with the weight function on
$K$-varieties}\label{ss-trick} \sss The weight function
$\weight_{\mathcal{I}}$ can also be compared to the one for
$K$-varieties and differential forms in \eqref{sss-weightext}, as
follows. We set $R=F\llbracket t\rrbracket$ and $K=F\llpar t\rrpar$. Let $X$ be a connected
smooth $F$-variety and let $\mathcal{I}$ be a coherent ideal sheaf
on $X$.
%
 We denote by $n+1$ the
dimension of $X$.
    Let $h:Y\to X$ be a log resolution of $\mathcal{I}$.
 We write $Z(\mathcal{I}\mathcal{O}_Y)=\sum_{i\in I}N_iE_i$ and
 $K_{Y/X}=\sum_{i\in I}(\mu_i-1)E_i$.
 For every point $\xi$ of $Z(\mathcal{I}\mathcal{O}_Y)$, we can find an open neighbourhood  $U$ of $\xi$ in $Y$ and a regular function $f$ on $U$ that
   generates
the ideal sheaf $\mathcal{I}\mathcal{O}_U$.  We view $U$ as a
$F[t]$-scheme by means of the morphism
 $$f\colon U\to \A^1_k=\Spec F[t]$$ and we denote by $\mathscr{U}$
 the $R$-scheme obtained by extension of scalars. Note that $\mathscr{U}$ is an $sncd$-model for its generic fiber $\mathscr{U}_K$. We choose a
 volume form $\phi$ on some open neighbourhood $V$ of $h(\xi)$ in
 $X$, that is, a nowhere vanishing differential form of
 degree $n+1$. Shrinking $U$, we may assume that $U\subset h^{-1}(V)$ and that $U\smallsetminus Z(f)$ is smooth over $V$.

 \sss The
 differential form $h^*\phi$ on $U$ induces a relative volume form
$\omega$ in $\Omega^n_{U/V}(U\smallsetminus Z(f))$, uniquely
characterized by the property that $\omega\wedge df=h^*\phi$ in
$\Omega^{n+1}_{U/V}(U\smallsetminus Z(f))$. The form $\omega$ is
 called the Gelfand-Leray form
 associated to $f$ and $h^*\phi$; see for instance \cite[9.5]{NiSe}.
 It induces a volume form on $\cU_K$ that we denote again by
 $\omega$.
 By the adjunction formula \eqref{eq-adj}, there exists a unique
 isomorphism  $$\varphi\colon \omega_{U/F[t]}\to
 \omega_{U/F}=\Omega^{n+1}_{U/F}$$ of line bundles on $U$ such that
$\varphi(\theta)=\theta\wedge df$ for every open subset $W$ of
$U\smallsetminus Z(f)$ and every $\theta\in
\omega_{U/V}(W)=\Omega^{n}_{U/V}(W)$. Since canonical sheaves are
compatible with the flat base change $F[t]\to R$, it follows that
\begin{equation}\label{eq-GL}
\mathrm{div}_{\mathscr{U}}(\omega)=\sum_{i\in
 I}(\mu_i-1)(E_i\cap U).\end{equation}

\sss  The $K$-analytic space
 $\widehat{\mathscr{U}}_\eta$ can be identified with the subspace
 of $$\widehat{U}_\eta:=\red^{-1}_{\widehat{Y}}(U\cap Z(\mathcal{I}\mathcal{O}_Y))\subset \widehat{X}_\eta$$ consisting of the points $x$ where
 $|f(x)|=|t|_K=1/e$, by \cite[4.2]{ni-sing}. This embedding has a retraction
$$r_{\mathscr{U}}\colon \widehat{U}_\eta\to \widehat{\mathscr{U}}_\eta,\, x\mapsto
r_{\mathscr{U}}(x)$$
 where $r_{\mathscr{U}}(x)$ is the unique point of
 $\widehat{\mathscr{U}}_\eta$ such that
 $\red_{\widehat{\mathscr{U}}}(r_{\mathscr{U}}(x))=\red_{\widehat{Y}}(x)$ and
 $$|g(r_{\mathscr{U}}(x))|=|g(x)|^{-1/\ln |f(x)|}$$ for all $g$ in
 $\mathcal{O}_{Y,\red_{\widehat{Y}}(x)}$.
  One can easily deduce from
 \eqref{eq-GL} that the restriction of $\weight_{\mathcal{I}}$ to $\widehat{U}_\eta$
 coincides  with $\weight_{\omega} \circ r_{\mathscr{U}}$, where $\weight_{\omega}$ is the weight function
     associated to $(\cU_K,\omega)$.


\subsection{An analog of the Kontsevich-Soibelman
skeleton}\label{subsec-KSsk-birat}
 \sss In analogy with the
definition of $\omega$-essential divisorial points in
\eqref{sss-KSsk}, we call a divisorial valuation $v$ as in
\eqref{sss-birat} $\mathcal{I}$-essential if
$$\weight_{\mathcal{I}}(v)=\lct(X,\mathcal{I}),$$
that is, if the divisor $E$ computes the log canonical threshold.
In that case, we also say that the divisor $E$ is
$\mathcal{I}$-essential.
 We define the skeleton $\Sk(X,\mathcal{I})$ of the pair
$(X,\mathcal{I})$ as the closure in $\widehat{X}_\eta$ of the set
of $\mathcal{I}$-essential divisorial valuations.  It is
well-known that, when $h\colon Y\to X$ is a log resolution of
 $\mathcal{I}$, the skeleton $\Sk(X,\mathcal{I})$ is the union of
 the open faces of $\Sk(\widehat{Y})$ corresponding to connected components of
 intersections of $\mathcal{I}$-essential prime components of
 $\mathcal{I}\mathcal{O}_Y$. This is the analog of Theorem
 \ref{thm-KS}. In particular, the subspace $\Sk(X,\mathcal{I})$ of
 $\Sk(\widehat{Y})\subset \widehat{X}_\eta$ is independent of the log resolution
 $h$. One can endow it with a natural piecewise affine structure
 as in Section \ref{ss-pwaff}.

 \sss \label{sss-shoko} The Shokurov-Koll\'ar Connectedness Theorem \cite[17.4]{kollar} implies that for every
 point $x$ in $Z(\mathcal{I})$ and every sufficiently small open
 neighbourhood $U$ of $x$ in $Z(\mathcal{I})$, the inverse image
 of $U$ under
 the reduction map
 $$\red_{\widehat{X}}\colon \Sk(X,\mathcal{I})\to Z(\mathcal{I})$$ is connected.
 Thus our Connectedness  Theorem for skeleta of $K$-varieties of geometric genus one
 (Corollary \ref{cor-connected}) can be viewed as an analog of the Shokurov-Koll\'ar
Connectedness
 Theorem.

\end{document}